\documentclass[letterpaper, 10pt, dvipsnames]{article}

\usepackage{bold-extra}
\usepackage[utf8]{inputenc}    
\usepackage[T1]{fontenc}
\usepackage[english]{babel}
\usepackage[normalem]{ulem}
\usepackage{dsfont, bm, pifont, mathrsfs}
\usepackage{stmaryrd}
\usepackage{comment}
\usepackage{bbm}
\usepackage{enumitem}
\usepackage{float, graphicx, caption, wrapfig, tikz}
\usepackage[margin=30pt]{subcaption}

\usepackage{amsthm, amsmath, amsfonts, amssymb, mathrsfs, mathtools}
\usepackage[alphabetic]{amsrefs}

\usepackage{hyperref, xcolor, titlesec} 
\hypersetup{colorlinks = true, urlcolor = RoyalBlue!70!Black,linkcolor=BrickRed, citecolor=PineGreen, bookmarksopen = true}
\usepackage[nomarginpar]{geometry}
\numberwithin{equation}{section}

\usepackage{cleveref}
\renewcommand{\ge}{\geqslant}
\renewcommand{\le}{\leqslant}

\newcommand{\nc}{\normalcolor}

\let \d \relax
\newcommand{\d}{\mathrm{d}}
\newcommand{\I}{\mathrm{i}}
\newcommand{\scp}[2]{\langle #1,#2\rangle}

\let\O\relax
\newcommand{\O}[1]{\mathcal{O}\left(#1\right)}
\newcommand{\E}[1]{\mathds{E}\left[#1\right]}
\newcommand{\EL}[1]{\mathds{E}\left[#1\middle\vert\bm{\lambda}\right]}

\newcommand{\unn}[2]{[\![#1,#2]\!]}
\let\Im\relax
\DeclareMathOperator{\Im}{Im}
\let\Re\relax
\DeclareMathOperator{\Re}{Re}
\DeclareMathOperator{\Tr}{Tr}

\renewcommand{\epsilon}{\varepsilon}
\renewcommand{\tilde}{\widetilde}

\def\bet{\begin{thm}}
\def\eet{\end{thm}}
\def\bel{\begin{lem}}
\def\eel{\end{lem}}
\def\bas{\begin{ass}}
\def\eas{\end{ass}}
\def\bec{\begin{cor}}
\def\eec{\end{cor}}
\def\bed{\begin{defn}}
\def\eed{\end{defn}}
\def\bep{\begin{prop}}
\def\eep{\end{prop}}
\def\beq{\begin{equation}}
\def\eeq{\end{equation}}
\def\bea{\begin{equation*}}
\def\eea{\end{equation*}}
\def\bex{\begin{ex}}
\def\eex{\end{ex}}
\def\bp{\begin{proof}}
\def\ep{\end{proof}}

\def\1{{\mathbbm 1}}

\def\benr{\begin{enumerate}[label=(\roman*)]}
\def\eenr{\end{enumerate}}

\def\N{\mathbb{N}}

\def\R{\mathbb{R}}
\def\C{\mathbb{C}}

\def\E{\mathbb{E}}

\newcommand\x{\mathbf{x}}
\newcommand\y{\mathbf{y}}

\renewcommand{\u}{\mathbf{u}}

\newcommand{\bma}{\begin{bmatrix}}
\newcommand{\ema}{\end{bmatrix}}

\def\phi{\varphi}

\renewcommand{\hat}{\widehat}

\newtheorem{ccounter}{ccounter}[section]
\newtheorem{thm}[ccounter]{Theorem}
\newtheorem{lem}[ccounter]{Lemma}
\newtheorem{cor}[ccounter]{Corollary}
\newtheorem{defn}[ccounter]{Definition}
\newtheorem{prop}[ccounter]{Proposition}
\newtheorem{ass}[ccounter]{Assumption}
\newtheorem{ex}[ccounter]{Example}

\theoremstyle{definition}
\newtheorem{rmk}[ccounter]{Remark}

\titleformat{\paragraph}[runin]{\itshape\normalsize}{\theparagraph}{}{}
\titleformat{\subparagraph}[runin]{\itshape\normalsize}{\theparagraph}{0em}{}
\titleformat{\section}[block]{\normalfont\filcenter}{\Large\thesection .}{.7em}{\Large\scshape}
\titleformat{\subsection}[runin]{\normalfont}{\large \bf \thesubsection .}{.5em}{\large\bf}[.]
\titleformat{\subsubsection}[runin]{\normalfont}{\bf \thesubsubsection .}{.5em}{\bf}[.]

\usepackage{tocloft}

\begin{document}

\title{\scshape\bfseries{Fluctuations of eigenvector overlaps and the Berry conjecture for Wigner matrices}}
\date{}
\author{L. \textsc{Benigni}\\\vspace{-0.15cm}\footnotesize{\it{Université de Montréal}}\\\footnotesize{\it{lucas.benigni@umontreal.ca}}\and G. \textsc{Cipolloni}\\\vspace{-0.15cm}\footnotesize{\it{Princeton University}}\\\footnotesize{\it{gc4233@princeton.edu}}}
\maketitle

\begin{abstract}
We prove that any finite collection of quadratic forms (overlaps) of general deterministic matrices and eigenvectors of an $N\times N$ Wigner matrix has joint Gaussian fluctuations. This can be viewed as the random matrix analogue of the Berry random wave conjecture.
\end{abstract}

\section{Introduction}
\tikzset{every node/.style={circle, inner sep = .2em}}

The groundbreaking discovery of Wigner \cite{wigner1957report} was that gap statistics of energy levels of heavy nuclei are universal and can be modeled by eigenvalue statistics
of large random Hermitian matrices called \emph{Wigner matrices}. We recall that such matrices are $N\times N$ Hermitian matrices $W^*=W$ with centered i.i.d. entries
(modulo symmetry) of finite variance. While the analysis of eigenvectors were not considered in Wigner's original study, they were examined by Porter and Thomas \cite{porter1956fluctuations}
to understand nuclear reaction widths. In general, $W$ is a tractable mathematical model for a Hamiltonian of a disordered quantum system and in particular, its eigenvectors, as stationary states of the system, are important quantities to understand. 

Let $\{\u_i\}_{i=1}^N$ denote the $\ell^2$-normalized eigenvectors of $W$. Several properties of $\u_i$ can be examined to better understand disordered quantum systems. \emph{Delocalization} consists in bounding their individual entries and it can be shown \cites{erdos2009semicircle,erdos2009local,vu2015random,benigni2022optimal} that $\sqrt{N}\Vert \u_i\Vert_\infty \leqslant C\sqrt{\log N}$ with very high probability\footnote{ We say that an event $\Omega$ holds with very high probability if $\mathbb{P}(\Omega^c)\le N^{-D}$ for any $D>0$.}. Proving a \emph{Porter--Thomas law} consists in showing the convergence of $\sqrt{2N}\u_i(\alpha)$ to a standard normal variable; this has been proved in the series of work \cites{knowles2013eigenvector, tao2012random, bourgade2013eigenvector, marcinek2020high}. For a deterministic \emph{traceless} matrix $A$, i.e. $\mathrm{Tr}A=0$, the \emph{Eigenstate Thermalization Hypothesis} (ETH), introduced in \cite{deutsch2018eigenstate}, consists in examining the quadratic form $\scp{\u_i}{A\u_j}$ and proving that with very high probability and any $\xi>0$,
\beq\label{eq:eth}
\left\vert \scp{\u_i}{A\u_j}\right\vert\leqslant N^\xi \sqrt{\frac{\langle \vert A \vert^2\rangle}{N}}.
\eeq
Here $\langle \cdot \rangle:= N^{-1}\mathrm{Tr}[\cdot]$ denotes the normalized trace. For Wigner matrices the ETH conjecture was proven in \cite{cipolloni2022rank}. Partial previous results for specific observables $A$ were given in \cites{bourgade2013eigenvector, bourgade2018random, cipolloni2021eigenstate, benigni2022fluctuations}. 

The universal properties of Wigner matrices for truly interacting quantum systems are yet to be proved, but it is conjectured to be true for many different models. For instance, in quantum chaos theory, the celebrated Bohigas--Giannoni--Schmit conjecture \cite{bohigas1984characterization} states that for a classically chaotic quantum billiard, eigenvalue statistics of the Laplacian should be given by random matrix statistics.  The comparison does not stop at eigenvalues since the ETH can also be compared with the Quantum Unique Ergodicity (QUE) conjecture introduced by Rudnick and Sarnak \cite{rudnick1994behaviour} which states that if $\{\psi_k\}_{k\ge 1}$ is an orthonormal basis of eigenfunctions (with increasing eigenvalues) of the Laplace--Beltrami operator on a negatively curved manifold $(\mathcal{M},\mu)$ then for any open set $A\subset \mathcal{M}$
\[
\int_A \vert \psi_k(x)\vert^2\mu(\mathrm{d}x)\xrightarrow[k\to\infty]{}\mu(A).
\]
Quantum ergodicity, which proves this convergence along a subsequence, has been proved in greater generality in \cites{shnirel1974ergodic, zelditch1987uniform, de1985ergodicite, anantharaman2019quantum}, while QUE has only been proved for some arithmetic surfaces in \cites{lindenstrauss2006invariant, holowinsky2010mass, holowinsky2010sieving} using advanced tools from number theory and ergodic theory. For negatively curved manifolds, Berry's \emph{random wave conjecture} from \cite{berry1977regular} asserts that the wave functions $\psi_i$, as $i\to\infty$, tend to exhibit Gaussian random behavior. In particular, they should locally ``\emph{converge}'' to the \emph{isotropic Gaussian field}, a centered stationary Gaussian field $\Psi$ on $\mathbb{R}^d$ with covariance
\[
\mathds{E}\left[
	\Psi(x)\Psi(y)
\right] = \int_{\mathds{S}^{d-1}}\mathrm{e}^{\mathrm{i}(x-y)\cdot \theta}\mathrm{d}\theta.
\] 
Recently, \cites{ingremeau2021local, abert2018eigenfunctions} gave a rigorous statement of the Berry conjecture in this setting.

\begin{figure}[!ht]
	\centering
	\includegraphics[width =.49\linewidth]{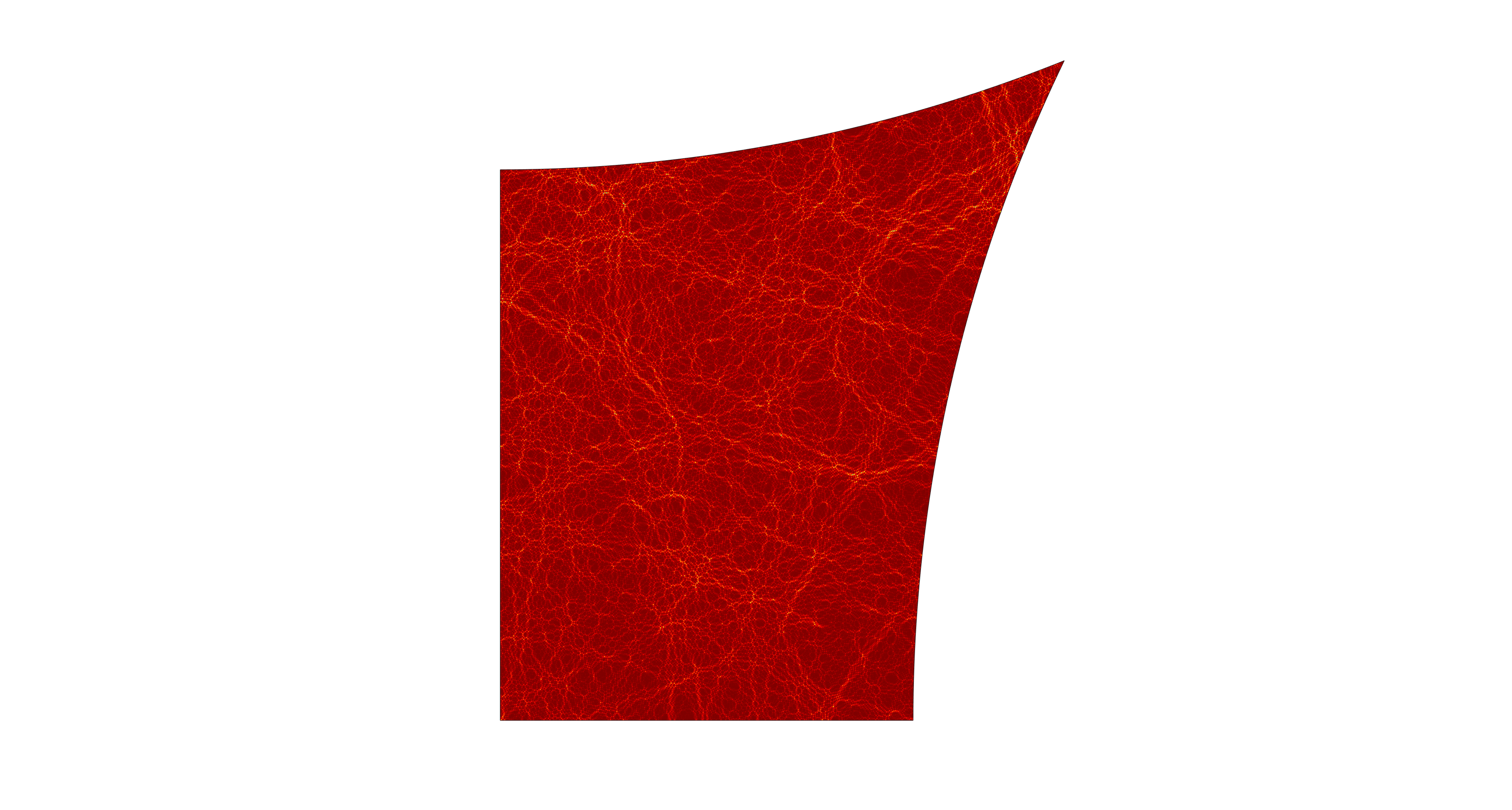}
	\includegraphics[width =.49\linewidth]{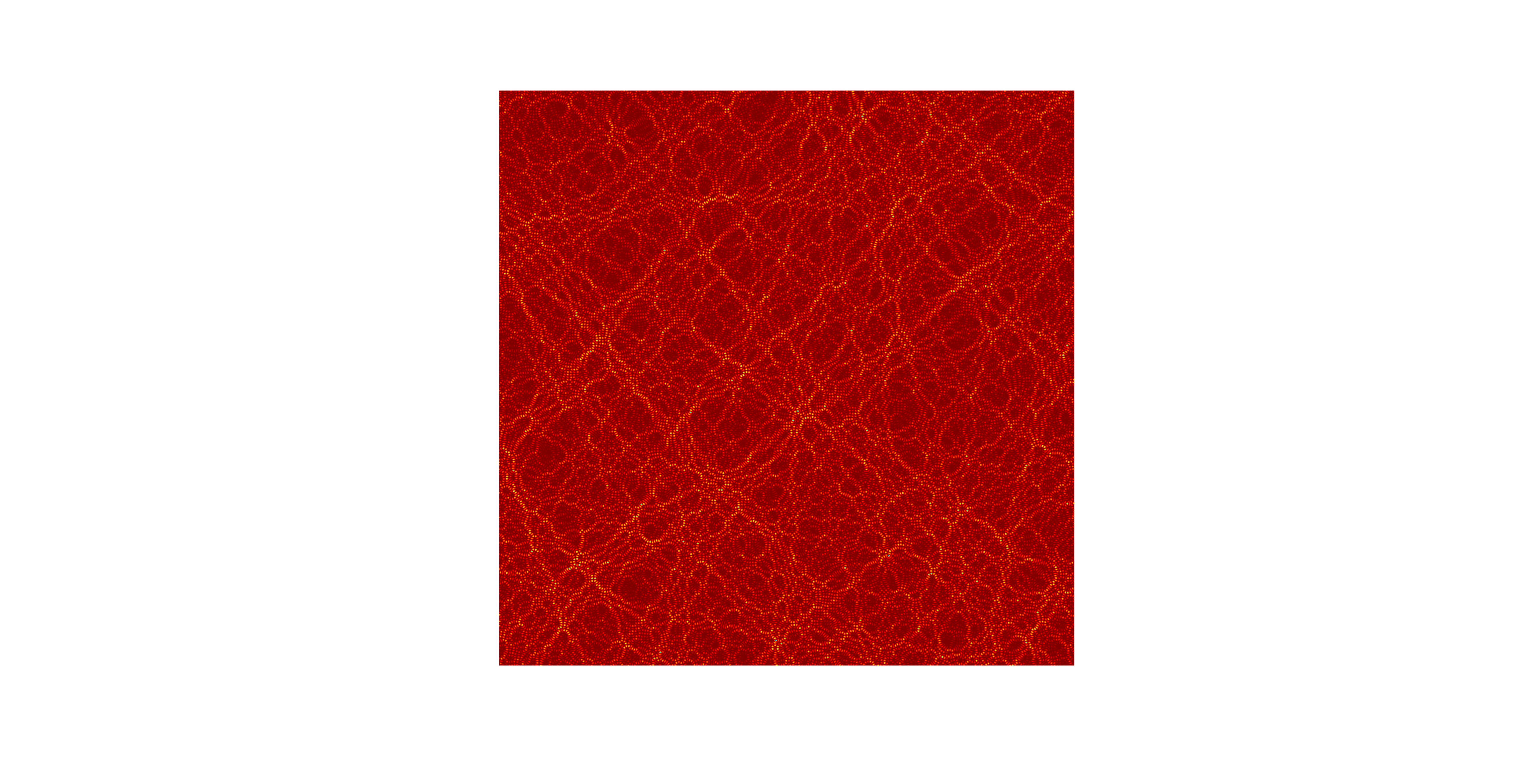}
	\caption{High energy eigenfunction of the Laplacian on the ``Barnett stadium'' which is conjectured to be quantum unique ergodic and the random plane wave on a square.}
\end{figure}

In this article, we study joint fluctuations of the overlaps $\scp{\u_i}{A\u_j}$. For simplicity, consider $A$ and $B$ to be two deterministic \emph{traceless} Hermitian matrices. The first results in this direction are in \cite{benigni2021fermionic} where it was shown that for $A=\sum_{\alpha\in I} | \mathbf{q}_\alpha\rangle\langle \mathbf{q}_\alpha| - \vert I\vert/N$, with $(\mathbf{q}_\alpha)$ being an orthonormal family of vectors in $\ell^2$ and $1\ll \vert I\vert\ll N$,
\[
\mathds{E}\left[\frac{N^2}{\vert I\vert}\scp{\u_i}{A\u_i}\scp{\u_j}{A\u_j}\right]\ll 1
\quad\text{and}\quad \mathds{E}\left[\frac{N^2}{\vert I\vert}\scp{\u_i}{A\u_j}^2\right]\xrightarrow[N\to\infty]{}{1}.
\]
The main difficulty for this result at the time was to understand correlations between different entries of different eigenvectors and was handled by introducing new moment observables along the Dyson Brownian motion.  For the same specific family of observables, and introducing again other symmetrized observables, convergence of the diagonal overlap to the Gaussian distribution was proved in \cite{benigni2022fluctuations} for all eigenvectors (including the edge). Simultaneously, for general full-rank observables $A$ and indices in the bulk, the same convergence was proved in \cite{cipolloni2022normal} using multi-resolvent local laws proved in \cite{cipolloni2021eigenstate},
\[
\sqrt{\frac{N}{\langle \vert A\vert^2\rangle}}\scp{\u_i}{A\u_i}\xrightarrow[N\to\infty]{}\mathcal{N}(0,1).
\]
This convergence was then extended to all observables $A$ such that $A$ has effective rank $\gg 1$ in \cite{cipolloni2022rank} by refining these multi-resolvent local laws. Note that the condition is necessary to obtain some averaging in light of the central limit theorem making the result optimal. Our goal in this article is to obtain similar convergence, not just for the diagonal overlaps, but jointly for any finite collections of eigenvectors and observables $A$ using the machinery set up developed in \cite{marcinek2020high} and \cite{cipolloni2022normal}.

We prove that $(\langle \u_i,A \u_j\rangle)_{i,j,A}$, for bulk indices $i,j$ and deterministic matrices $A$, forms a Gaussian field with covariance structure given by (see Theorem~\ref{theo:main} below)

\begin{equation}
	\label{eq:covstruc}
	\E \left[ N \langle \u_i,A \u_j\rangle\langle \u_k,B \u_l\rangle\right]\approx 
	 \left[\delta_{kj}\delta_{li}+\left(\frac{2}{\beta}-1\right)\delta_{ki}\delta_{lj}\right] \langle AB\rangle,
\end{equation}
where $\beta$ is the Dyson index such that $\beta=1$ in the symmetric case and $\beta=2$ in the Hermitian case. In particular, note that $\langle \u_i,A \u_i\rangle,  \langle \u_j,B \u_j\rangle$ are asymptotically independent for any $j\ne i$ and, even more interestingly, that  $\langle \u_i,A \u_i\rangle,  \langle \u_i,B \u_i\rangle$ are also asymptotically independent as long as $\langle AB\rangle \ll \sqrt{\langle A^2\rangle\langle B^2\rangle}$. The limiting statement in \eqref{eq:covstruc} can be easily checked for $W$ drawn from the GOE/GUE ensembles by Weingarten calculus as the eigenbasis of these matrices are distributed according to the Haar measure on the orthogonal or unitary group, respectively. Thus \eqref{eq:covstruc} can be seen as a universality statement for finite collection of eigenvectors of Wigner matrices: any finite number of eigenvectors $\u_1,\dots,\u_k$ behaves as if they were Haar distributed on the corresponding compact group. While the covariance structure is different due to eigenvectors being close to Haar distributed on unitary or othgonal group, our result proves joint Gaussian fluctuations around the QUE property of eigenvectors and thus can be understood as an equivalent to the Berry conjecture in the context of random matrices.

Our result, and many mentioned above, rely on a careful analysis of the Dyson vector flow, the flow of eigenvectors induced by the Dyson Brownian motion (DBM). This flow was first computed in \cite{bru1989diffusions} and first used in this setting in the seminal paper \cite{bourgade2013eigenvector} where it was used to compute asymptotic moments of individual eigenvector entries. The dynamical method to prove universality results as introduced in the groundbreaking paper \cite{erdos2011universality} consists in three steps. The first step establishes an (almost) optimal a priori bound for the quantity of interest and the second step use the DBM, combined with our a priori bound, to prove our joint fluctuations of eigenvectors along the dynamics. Finally, the third step is based on a comparison step between the dynamical model and Wigner matrices.

We now briefly describe the first and second steps; the third step is performed by an application of the by now standard \emph{Green's function comparison theorem} (see Appendix~\ref{sec:GFT} below). For our result, the a priori bound is the ETH \eqref{eq:eth} obtained in \cites{cipolloni2021eigenstate, cipolloni2022rank}.
The proof of \eqref{eq:eth} relies on the observation that
\[
\mathrm{Av}_{ij} N\big|\langle \u_i,{A} \u_j\rangle\big|^2\sim \frac{1}{N}\mathrm{Tr}[\Im G(z){A}\Im G(z'){A}],
\]
with $G(z):=(W-z)^{-1}$ denoting the resolvent of $W$ and $z,z'$ are possibly different spectral parameters properly chosen with $\Im z=\Im z'=N^{-1+\epsilon}$. Here, $\mathrm{Av}_{ij}$ denotes an averaging over $N^\epsilon$ indices. Then, to conclude \eqref{eq:eth}, it thus remains to obtain the desired upper bound for $\mathrm{Tr}[\Im G\mathring{A}\Im G' \mathring{A}]$ which can be handled by proving a very delicate \emph{multi--resolvent} local law; in particular, in \cites{cipolloni2021eigenstate, cipolloni2022rank, cipolloni2022optimal} it was proven that
\begin{equation}
\label{eq:bound2g2a}
\frac{1}{N}\mathrm{Tr}[\Im G{A}\Im G' {A}]\le \langle {A}^2\rangle.
\end{equation}
The fundamental observation to obtain \eqref{eq:bound2g2a} is that the resolvent $G$ is much smaller when tested against matrices with zero trace.

Given the a priori bound \eqref{eq:eth} as an input, to prove Theorem~\ref{theo:main}, we compute arbitrary moments of several eigenvector overlaps $\langle \u_i,A \u_j\rangle$. Following Marcinek--Yau \cite{marcinek2020high} one can see that arbitrary moments of eigenvectors overlaps, which are encoded by
\begin{equation}
\label{eq:ourobs}
f_{\bm{\lambda},t}({\bf x},{\bf A}):= f(\x)= \E\left[ \prod_{i=1}^{n} \langle {\bf u}_{x_{2i-1}},A_i {\bf u}_{x_{2i}}\rangle\Big|\bm{ \lambda}\right],
\end{equation}
 evolve according to the \emph{colored} eigenvector moment flow (see \eqref{eq:maineqob}--\eqref{eq:moveexchange} below). Then to study relaxation properties of this flow, we rely on energy estimates for multi--indexed DBM developed in \cite{marcinek2020high} for finite rank $A$'s, and extended to the analysis of general observables $A$ in \cite{cipolloni2022normal}. The observable considered in \cite{cipolloni2022normal} is different compared to \eqref{eq:ourobs} since the authors studied the evolution of \emph{perfect matching observables} introduced in \cite{bourgade2018random}, which enabled to obtain a central limit theorem only for diagonal overlaps. The main novelty in the present article is that the deterministic approximation of $f(\x)$ depends on the configuration $\x$ and $f$ follows the \emph{colored} eigenvector moment flow while the perfect matching observable evolve according to the original one from \cite{bourgade2013eigenvector}. The \emph{colored} version from \cite{marcinek2020high} loses its parabolicity due to the exchange term $\mathscr{E}$ in \eqref{eq:moveexchange}. In particular, we need to introduce the correct \emph{ansatz observable} $F(\x)$, the deterministic limit of $f(\x),$ and prove that it is stationary along the dynamics which is done in Subsection \ref{sub:ansatz}.

\section{Main result}

We consider $N\times N$ Wigner matrices, i.e. real symmetric or complex Hermitian matrices $W=W^*$ with entries distributed as $w_{ab}\stackrel{\d}{=}n^{-1/2}\chi_\mathrm{od}$, for $a>b$, and $w_{aa}\stackrel{\d}{=}n^{-1/2}\chi_\d$. On the random variables $\chi_\d,\chi_\mathrm{od}$ we assume:

\begin{ass}
\label{ass:momass}
$\chi_\mathrm{od} $ is a real or complex random variable with $\E \chi_\mathrm{od}=0$, $\E|\chi_\mathrm{od}|^2=1$; additionally, in the complex case we also assume that $\E \chi_\mathrm{od}^2=0$. $\chi_\d$ is a real random variable with $\E\chi_\d=0$.

Furthermore, we assume that all the moments of $\chi_\d, \chi_\mathrm{od}$ exist, i.e. that for any $p\in\N$ there exist constants $C_p>0$ such that
\begin{equation}
\label{eq:momass}
\E|\chi_{\mathrm{od}}|^p+\E|\chi_\d|^p\le C_p.
\end{equation}
\end{ass}

By $\lambda_1\le \dots\le \lambda_N$ we denote the eigenvalues of $W$, and by $\{\u_i\}_{i\in [N]}$ the corresponding orthonormal eigenvectors. Since the eigenvectors are defined only modulo a phase, we assume that the eigenvectors are always multiplied by a uniformly random phase $e^{\mathrm{i}\theta_i}\u_i$ where $\theta_i \sim \mathrm{Unif}([0,2\pi])$ in the Hermitian case and $\theta_i\sim\mathrm{Unif}(\{0,\pi\})$ in the symmetric case; this choice will simply the statement of our main result in Theorem~\ref{theo:main} below. 
 To make the notation more concise, given a deterministic matrix $A\in\C^{N\times N},$ we denote its \emph{traceless} part by $\mathring{A}:=A-\langle A\rangle$ where $\langle A\rangle =\frac{1}{N}\Tr A$.
  From now on we only consider the case when $A$ is a real symmetric or a complex Hermitian matrices, the general case is discussed below Theorem~\ref{theo:main}.

Before stating our main result we introduce the following rescaled quantity:
\begin{equation}
\label{eq:defphi}
\Phi_N(A,i,j):=\sqrt{\frac{N}{\langle \mathring{A}^2\rangle}}\big[\langle  e^{\mathrm{i}\theta_i}   \u_i, A  e^{\mathrm{i}\theta_j}  \u_j\rangle-\langle A\rangle\delta_{ij}\big].
\end{equation}

\begin{thm}
\label{theo:main}
Let $W$ be a Wigner matrix satisfying Assumption~\ref{ass:momass}. Fix any $p\in\N$, any small $\delta,\delta'>0$, and any real symmetric or complex Hermitian deterministic matrices $A_1,\dots, A_p$ with $\lVert A_i\rVert\le 1$, $\langle \mathring{A_i}^2\rangle\ge N^{-1+\delta'}$. Then, for any $i_1,\dots, i_{2p}\in \unn{\delta N}{(1-\delta)N}$, the collection $(\Phi_N(A_1,i_1,i_2),\dots, \Phi_N(A_p,i_{2p-1},i_{2p}))$ is approximated in the sense of moments\footnote{ For two possilby $N$-dependent random variables $X_N$, $Y_N$ we say that $X_N$ is approximated in the sense of moments by $Y_N$ if $\E |X_N|^k=\E |Y_N|^k+\mathcal{O}(N^{-c(k)})$, with $c(k)>0$ a small constant depending on $k$.} by a Gaussian field
\begin{equation}
\label{eq:convmult}
 \Big(\Phi(A_1,i_1,i_2),\dots, \Phi(A_p,i_{2p-1},i_{2p})\Big),
\end{equation}
with covariance structure
\begin{equation}
\E \Phi(A,i,j)\Phi(B,k,l)= \left[\delta_{kj}\delta_{li}+\left(\frac{2}{\beta}-1\right)\delta_{ki}\delta_{lj}\right]\frac{\langle \mathring{A}\mathring{B}\rangle}{\sqrt{\langle \mathring{A}^2\rangle\langle \mathring{B}^2\rangle}}.
\end{equation}
\end{thm}

The assumption $\langle \mathring{A_i}^2\rangle\ge N^{-1+\delta'}$ in Theorem~\ref{theo:main} is not a restriction since it ensures that the the observables $A_i$ are not finite rank, and for finite rank matrices \eqref{eq:convmult} does not hold; the behavior of finitely many entries has been studied in \cites{bourgade2013eigenvector, marcinek2020high}. Furthermore, in Theorem~\ref{theo:main} we stated the convergence \eqref{eq:convmult} only for real symmetric or complex Hermitian matrices $A$. The general case immediately follows by polarization and by the fact that any matrix can be written as the sum of symmetric or Hermitian matrices.

\begin{rmk}[Special cases of Theorem~\ref{theo:main}]

We now specialize Theorem~\ref{theo:main} to some simple relevant cases:

\begin{itemize}

\item[(i)] By \eqref{eq:convmult} for $p=1$ and $i\ne j$ it follows that (here $\mathbb{F}=\R$ for $\beta=1$ and $\mathbb{F}=\C$ for $\beta=2$)
\[
\sqrt{\frac{ N}{\langle \mathring{A}^2\rangle}}\big|\langle  \u_i, A   \u_j\rangle\big|\Longrightarrow \big|\mathcal{N}_\mathbb{F}(0,1)\big|.
\]

\item[(ii)]  By \eqref{eq:convmult} for $p=2$ it follows that, after the proper rescaling, $\langle \u_i,A\u_i\rangle, \langle \u_j,A\u_j\rangle$ are asymptotically independent for any $i\ne j$. Interestingly, we find out that $\langle \u_i,A\u_i\rangle, \langle \u_i,B\u_i\rangle$ are also asymptotically independent as long as
\[
\langle \mathring{A}\mathring{B}\rangle \ll \sqrt{\langle \mathring{A}^2\rangle \langle \mathring{B}^2\rangle}.
\]

\end{itemize}

\end{rmk}

\section{Short-time Relaxation}
In this section we analyze the mixed moments of \eqref{eq:defphi} for different indices and observable matrices through the matrix Dyson Brownian motion. This dynamics is defined as 
\begin{equation}\label{eq:DBM}
	\d W_t=\frac{\d B_t}{\sqrt{N}},\qquad W_0=W,
\end{equation}
and the eigenvalues and eigenvectors of $W_t$ are the solutions to a system of couple stochastic differential equations \cite{bourgade2013eigenvector}*{Definition 2.2}.
It was observed in the seminal paper \cite{bourgade2013eigenvector} that  moments of eigenvectors follow a dynamics which can be represented as a random walk in a random environment given by the eigenvalue process of $W_t$. We now introduce the set of observables we will analyze which follow the \emph{colored} eigenvector moment flow, a variant of the process from \cite{bourgade2013eigenvector}, introduced in \cite{marcinek2020high}. 

Consider $\Lambda^n\subset [N]^{2n}$ be defined as
\[
\Lambda^n:=\left\{ {\bf x}\in[N]^{2n}\,:\, n_i({\bf x})\,\, \mathrm{even}\,\, \forall \, i\in [N]\right\}.
\]
For ${\bf x}\in[N]^{2n}$ we denote its entries by $x_i$, with $i\in [2n]$.

\begin{figure}[!ht]
	\centering
	\begin{tikzpicture}[scale = 1.2]
		\draw[-,black] (-1,0) -- (4,0);
		\node[draw,black,circle,fill=white] at (0,0) {};  
		\node[draw,black,circle,fill=black,label={below: \nc$j$}] (1) at (.5,0) {};
		\node[draw,black,circle,fill=black] (2) at (.5,0.3) {};
		\node[draw,black,circle,fill=black] (3) at (.5,.6) {}; 
		\node[draw,black,circle,fill=black] (4) at (.5,.9) {};
		\node[draw,black,circle,fill=white] at (1,0) {};
		\node[draw,black,circle,fill=black,label=below: \nc$k$] (5) at (1.5,0) {};
		\node[draw,black,circle,fill=black] (6) at (1.5,0.3) {};
		\node[draw,black,circle,fill=black,label=below: \nc $\ell$] (7) at (2,0) {};
		\node[draw,black,circle,fill=black] (8) at (2,.3) {};
		\node[draw,black,circle,fill=white] at (2.5,0) {};
		\node[draw,black,circle,fill=white] at (3,0) {};
	\end{tikzpicture}
	\caption{A visualization of the vertices of the configuration $\x = (j,j,j,j,k,\ell,k,\ell)$ corresponding to the conditional moment of $\langle \u_{j},A_1\u_j\rangle\langle \u_j,A_2\u_j\rangle\langle \u_k,A_3\u_\ell\rangle\langle \u_k,A_4\u_\ell\rangle$.}
	\label{fig:config}
\end{figure}
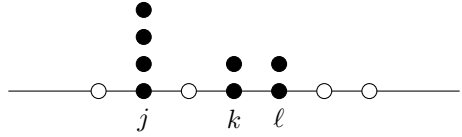

\bed For a given family of traceless symmetric matrices $\mathbf{A}=(A_i)_{1\leqslant i\leqslant n}$, we define
\begin{equation}
\label{eq:observ}
f_{\bm{\lambda},t}({\bf x},{\bf A}):= f(\x)= \E\left[ \prod_{i=1}^{n} \langle {\bf u}_{x_{2i-1}},A_i {\bf u}_{x_{2i}}\rangle\Big|\bm{ \lambda}\right]
\end{equation}
where the conditioning is on the whole path of eigenvalues from $0$ to $1$.
\eed

This observable follow the colored eigenvector moment flow from \cite{marcinek2020high}.

\bet[\cite{marcinek2020high}*{Theorem 4.8}] It follows that
\begin{equation}
\label{eq:maineqob}
\partial_t f_t({\bf x})=\sum_{1\le i< j\le N}c_{ij}(t)\mathscr{L}_{ij}f_t({\bf x}), \qquad  c_{ij}(t):= \frac{1}{N(\lambda_i-\lambda_j)^2},
\end{equation}
where $\mathscr{L}_{ij}=\mathscr{M}_{ij}-\mathscr{E}_{ij}$, with
\begin{equation}
\begin{split}
\label{eq:moveexchange}
\mathscr{M}_{ij}f({\bf x})&=\frac{n_j({\bf x})+1}{n_i({\bf x})-1}\sum_{a\neq b}\left(f(m^{ij}_{ab}{\bf x})-f({\bf x})\right)+\frac{n_i({\bf x})+1}{n_j({\bf x})-1}\sum_{a\neq b}\left(f(m^{ji}_{ab}{\bf x})-f({\bf x})\right) \\
\mathscr{E}_{ij}f({\bf x})&=\sum_{a\neq b}\left(f(s_{ab}^{ij}{\bf x})-f({\bf x})\right).
\end{split}
\end{equation}
Here we used the defnition
\begin{equation}
m_{ab}^{ij}{\bf x}={\bf x}+\bm1_{x_a=x_b=i}(j-i)({\bf e}_a+{\bf e}_b), \qquad s_{ab}^{ij}{\bf x}={\bf x}+\bm1_{x_a=i}\bm1_{x_b=j}(j-i)({\bf e}_a-{\bf e}_b).
\end{equation}
Besides, the reversible measure for the generator $\mathscr{L}$ is given by
\[
\pi(\x)=\prod_{i=1}^N \left[(n_i(\x)-1)!!\right]^2
\]
\eet
\subsection{Ansatz Observable}\label{sub:ansatz}
We now define the observable $F(\x,\mathbf{A})$ which corresponds to the asymptotic deterministic behavior of $f(\x,\mathbf{A}).$ It is the analogue of the \emph{anzats observable} from \cite{marcinek2020high}. In particular since we want to show that the family of observables of the form $\scp{\u_i}{A_j\u_k}$ form a Gaussian process with a specific covariance structure, we are going to use Wick's theorem to describe $F(\x)$.

From \eqref{eq:observ}, we see that the configuration of particles is not enough to describe $f(\x,\mathbf{A})$ since we need to add an ordering of the particles for it to be well defined. This can correspond to performing a perfect matching on the configuration $\x$ (note that we consider configurations such that $n_i(\x)$ is even for all $i$). Thus, we can write $f(\x,\mathbf{A})=f(\mathcal{G},\mathbf{A})$ where $\mathcal{G}=(\mathcal{V},\mathcal{E})$. with $\mathcal{V}=\{x_i\}$ and $\mathcal{E}=\{\{x_{2i-1},x_{2i}\}\}$. Thus we rewrite,
\beq\label{eq:graph}
f(\mathcal{G},\mathbf{A})=\EL{\prod_{e=\{k,\ell\}\in\mathcal{E}}\scp{\u_k}{A_e\u_\ell}}
\eeq
where the reindexing of the matrices from $\mathbf{A}$ is straightforward. For each edge $e\in\mathcal{E}$, we denote $n(e)$ the number of times it occurs in the graphs,
\[
n(\{k,\ell\})=\left\vert
\{
e\in \mathcal{E},e=\{k,\ell\}
\}
\right\vert,
\]
and we denote by $e^{(1)},\dots,e^{(n(e))}$ the occurrences of the edges in the graph, note that it is important to differentiate them since while they connect the same eigenvector indices, they may be labeled by different matrices from $\mathbf{A}$. We define an equivalence relation between edges $e \sim e'$ if and only if $e=e'$ so that we can define the set $\tilde{\mathcal{E}} = \mathcal{E}/\sim$ where each element is a representative of the corresponding edge. The graph representation is illustrated in \Cref{fig:graph}

\begin{figure}[!ht]
	\centering
	\begin{tikzpicture}[scale = 1.2]
		\draw[-, black] (-1,0) -- (4,0);
		\node[draw,black,circle,fill=white] at (0,0) {};  
		\node[draw,black,circle,fill=black,label={below:}] (1) at (.5,0) {};
		\node[draw,black,circle,fill=black] (2) at (.5,0.5) {};
		\node[draw,black,circle,fill=black] (3) at (.5,1) {}; 
		\node[draw,black,circle,fill=black] (4) at (.5,1.5) {};
		\node[draw,black,circle,fill=white] at (1,0) {};
		\node[draw,black,circle,fill=black,label=below left:] (5) at (1.5,0) {};
		\node[draw,black,circle,fill=black] (6) at (1.5,0.5) {};
		\node[draw,black,circle,fill=white] at (2,0) {};
		\node[draw,black,circle,fill=black] (8) at (2.5,.5) {};
		\node[draw,black,circle,fill=black] (7) at (2.5,0) {};
		\node[draw,black,circle,fill=black] (9) at (2.5,1) {};
		\node[draw,black,circle,fill=black] (10) at (1.5,1.5) {};
		\node[draw,black,circle,fill=black] (11) at (1.5,1) {};
		\node[draw,black,circle,fill=black] (12) at (2.5,1.5) {};
		\node[draw,black,circle,fill=white] at (3,0) {};
		
		\draw[-, line width=.2em, bend left=60,black] (1) edge node[draw=none, midway, label={[label distance=-.2cm]left:$A_1$}] {} (3);
		\draw[-, line width=.2em, bend left=60,BrickRed] (4) edge node[draw=none, midway, label={[label distance=-.2cm]right:$A_2$}] {} (2);  
		\draw[-, line width=.2em,ForestGreen] (5) edge[bend right=40] node[draw=none, midway, label={[label distance =-.3cm]below:$A_3$}] {} (7); 
		\draw[-, line width=.2em,RoyalBlue] (6) edge node[draw=none, midway, label={[label distance =-.25cm]below:$A_4$}] {} (8); 
		\draw[-, line width=.2em,Violet] (11) edge node[draw=none, midway, label={[label distance =-.25cm]above:$A_5$}] {} (9); 		
		\draw[-, line width=.2em,BurntOrange] (12) edge[bend right=40] node[draw=none, midway, label={[label distance =-.3cm]above:$A_6$}] {} (10); 		
	\end{tikzpicture}
	\caption{An example of the graph representation of an ordered configuration. We have $n(\{j,j\})=2$ and $n(\{k,\ell\})=4$ and each edge is labeled (and thus colored) by a matrix from $\mathbf{A}.$ For this graph, the ansatz observable is given by $F(\x,\mathbf{A})=\frac{2}{27N^3}\langle A_1A_2\rangle\left(\langle A_3 A_4\rangle\langle A_5 A_6\rangle+\langle A_3 A_5\rangle\langle A_4 A_6\rangle+\langle A_3 A_6\rangle\langle A_4 A_5\rangle\right)$}
	\label{fig:graph}
\end{figure}
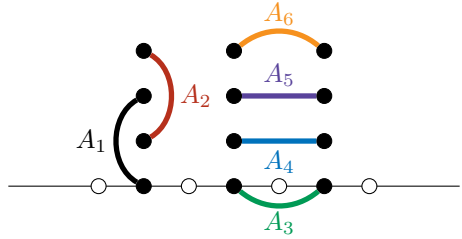

\bed
We define the \emph{ansatz observable} to be,
\[
F(\x,\mathbf{A})=F(\mathcal{G},\mathbf{A}) = \frac{1}{N^{n/2}\prod_{i=1}^N(n_i(\mathcal{G})-1)!!}\prod_{e=\{k,\ell\}\in\tilde{\mathcal{E}}}2^{\frac{n(e)}{2}\mathds{1}_{k=\ell}}\sum_{\Pi\in\mathrm{PM}[n(e)]}\prod_{\{i,j\}\in\Pi}\langle A_{e^{(i)}}A_{e^{(\Pi(i))}}\rangle
\]
where $n_i(\mathcal{G})=n_i(\mathcal{V})$ is the number of particles at site $i$, $\mathrm{PM}[i]$ is the set of perfect matchings on $\{1,\dots i\}$ with the convention that the empty sum is zero.
\eed

Before starting the analysis of the eigenvector observable, it is important to check that our ansatz observable $F$ is an actual one, in other words that it is in the kernel of $\mathscr{L}$. We first recall a corollary from \cite{marcinek2020high} which describes the kernel of $\mathscr{L}$.
\bel[\cite{marcinek2020high}*{Corollary 6.28}]
Consider a configuration $\x$ over $2n$ points, for a given perfect matching $\Pi$ over $2n$ points, we say that $\Pi \in \mathrm{Stab}(\x)$ if
\[
\Pi \bm{\cdot} \x \coloneqq (x_{\Pi(1)},\dots x_{\Pi(2n)})=(x_1,\dots,x_{2n})=\x.
\]
We then have
\[
\bigcap_{1\leqslant i<j\leqslant N}\mathrm{Ker}\,\mathscr{L}_{ij}=\mathrm{Span}\{\chi_{\Pi}\vert \Pi\in\mathrm{PM}[2n]\},
\quad
\chi_\Pi=\frac{\mathds{1}_{\Pi\in\mathrm{Stab}(\x)}}{\sqrt{\pi(\x)}} 
\]
\eel

Using this lemma we prove that $F$ is in the kernel of the operator $\mathscr{L}$:

\bel\label{lem:ansatz} It holds that
\[F\in \bigcap_{1\leqslant i<j\leqslant N}\mathrm{Ker}\,\mathscr{L}_{ij}\]
\eel

\begin{proof}
	The proof is based on rewriting the ansatz observable $F$ as a linear combination of $\chi_{\mathrm{\Pi}}.$ For now, it is defined as a product of sums of perfect matchings on the edges of $\mathcal{G}$. We first need to rewrite these perfect matchings as acting on the vertices of the graph rather than the edges. To do this, we use the first convention of ordered particles.
	
	For each edge $e$ of the graph, we consider $B_e$ to be the subgraph containing all and only occurrences of the edge $e$ from the initial graph. In particular, we see that $\mathcal{G}=\cup_{e\in\mathcal{E}}B_e$ where the union is not disjoint (for two occurrences of the same edge $e_1$ and $e_2$ we have $B_{e_1}=B_{e_2}$). The blocks of a configuration are illustrated in \Cref{fig:block}. 
	
	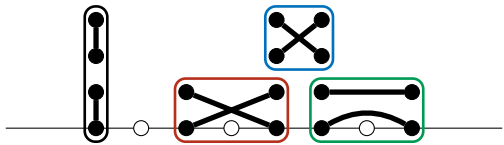
\begin{figure}[!ht]
		\centering
		\begin{tikzpicture}[scale = 1.2]
			\draw[-, black] (-1,0) -- (4.5,0);
			
			\node[draw,black,circle,fill=black] (1) at (0,0) {};
			\node[draw,black,circle,fill=black] (2) at (0,.4) {};
			\node[draw,black,circle,fill=black] (3) at (0,.8) {};
			\node[draw,black,circle,fill=black] (4) at (0,1.2) {};
			
			\node[draw,black,circle,fill=white] at (0.5,0) {};  
			
			\node[draw,black,circle,fill=black] (5) at (1,0) {};			
			\node[draw,black,circle,fill=black] (6) at (1,.4) {};	
			
			\node[draw,black,circle,fill=white] at (1.5,0) {};
			
			\node[draw,black,circle,fill=black] (7) at (2,0) {};			
			\node[draw,black,circle,fill=black] (8) at (2,.4) {};  	
			\node[draw,black,circle,fill=black] (9) at (2,.8) {};			
			\node[draw,black,circle,fill=black] (10) at (2,1.2) {};
			
			\node[draw,black,circle,fill=black] (11) at (2.5,0) {};			
			\node[draw,black,circle,fill=black] (12) at (2.5,.4) {};  	
			\node[draw,black,circle,fill=black] (13) at (2.5,.8) {};			
			\node[draw,black,circle,fill=black] (14) at (2.5,1.2) {};
			
			\node[draw,black,circle,fill=white] at (3,0) {};		
			
			\node[draw,black,circle,fill=black] (15) at (3.5,0) {};			
			\node[draw,black,circle,fill=black] (16) at (3.5,.4) {};	
			
			\draw[-, line width=.2em, black] (1) edge (2);
			\draw[-, line width=.2em, black] (3) edge (4);
			
			\draw[-, line width=.2em, black] (5) edge (8);
			\draw[-, line width=.2em, black] (6) edge (7);	
			
			\draw[-, line width=.2em, black] (9) edge (14);
			\draw[-, line width=.2em, black] (10) edge (13);
			
			\draw[-, line width=.2em, black] (11) edge[bend left] (15);
			\draw[-, line width=.2em, black] (12) edge (16);
			
			\draw[-, line width = .1em, black, rounded corners] (-.125,-.15) rectangle ++(.25,1.5);
			\draw[-, line width = .1em, BrickRed, rounded corners] (0.875,-.15) rectangle ++(1.25,.7);
			\draw[-, line width = .1em, RoyalBlue, rounded corners] (1.875,.65) rectangle ++(.75,.7);
			\draw[-, line width = .1em, ForestGreen, rounded corners] (2.375,-.15) rectangle ++(1.25,.7);	
		\end{tikzpicture}
		\caption{Illustration of blocks of a given configuration. All occurrences of a given edge are in the same block.}
		\label{fig:block}
	\end{figure}
	
	For a given edge $e=\{k,\ell\}$, there exist $j_1<j_2<\dots<j_{n(e)}$, using our first convention of ordered particles, such that 
	\[
	B_e = (x_{2j_1-1},x_{2j_1},x_{2j_2-1},x_{2j_2},\dots,x_{2j_{n(e)}-1},x_{2j_{n(e)}}).
	\]
	In particular, half of the $x$'s are $k$ and other half are $\ell$. From this notation, we define
	\[
	\mathcal{A}(B_e)=\mathds{1}_{n(e)\text{ is even}}\prod_{p=1}^{\frac{n(e)}{2}}\langle A_{j_{2p-1}}A_{j_{2p}}\rangle.
	\]
	We now define the set of perfect matchings which stabilize the block $B_e$ but also keep the edge structure, in other words $\mathcal{E}\text{-}\mathrm{Stab}(B_e)$ is the set of perfect matchings $\Pi$ over the set $\{2j_1-1,2j_1,\dots,2j_{n(e)}-1,2j_{n(e)}\},$
	such that $\Pi\in\mathrm{Stab}(B_e)$ and such that for all
	$p\in\{1,\dots,n(e)\}$ there exists
	$ q\in\{1,\dots,n(e)\}$ with $q\neq p$ such that
	$\Pi(\{{2j_p-1},{2j_p}\})
	=
	\{{2j_q-1},{2j_q}\}.$
	
	Let $r$ be the number of distinct blocks of the configuration $\x$ and denote them by $B_1,\dots B_r$. We define the set $\mathcal{G}\text{-}\mathrm{Stab}(\x)$ to be the set of perfect matchings over the whole configuration of $2n$ particles such that the restriction of $\Pi$ to each block $B_i$ is in $\mathcal{E}\text{-}\mathrm{Stab}(B_i).$
	
	The first observation is that for a given edge $e=\{k,\ell\}$ we have
	\[
	2^{\frac{n(e)}{2}\mathds{1}_{k=\ell}}\sum_{\Pi\in\mathrm{PM}[n(e)]}\prod_{i=1}^{n(e)}\langle A_{e^{(i)}}A_{e^{(\Pi(i))}}\rangle
	=
	\sum_{\Pi \in \mathcal{E}\text{-}\mathrm{Stab}(B_e)}\mathcal{A}\left(\Pi \bm{\cdot}B_e\right).
	\]
	Note that the $2^{\frac{n(e)}{2}}$ is necessary in the case when $e=\{k,k\}$ for some $k$. Since the stabilization property is immediate, there are now two choices for each matching between edges. Thus we can write,
	\begin{align*}
		F(\x,\mathbf{A}) = \frac{1}{N^{n/2}\sqrt{\pi(\x)}}\prod_{i=1}^r\sum_{\Pi\in\mathcal{E}\text{-}\mathrm{Stab}(B_i)}\mathcal{A}(\Pi\bm{\cdot}B_i)
		&=\frac{1}{N^{n/2}\sqrt{\pi(\x)}}\sum_{\substack{\Pi_1\in\mathcal{E}\text{-}\mathrm{Stab}(B_1)\\ \dots\\\Pi_r\in\mathcal{E}\text{-}\mathrm{Stab}(B_r)}}
		\prod_{i=1}^r \mathcal{A}(\Pi_i\bm{\cdot}B_i)\\
		&=
		\frac{1}{N^{n/2}}\sum_{\Pi\in\mathrm{PM}[2n]}\left[\prod_{i=1}^r \mathcal{A}(\Pi_{\vert B_i} \bm{\cdot}B_i)\right]
		\frac{\mathds{1}_{\Pi\in\mathcal{G}\text{-}\mathrm{Stab}(\x)}}{\sqrt{\pi(\x)}}
	\end{align*} 
where $\Pi_{\vert B_i}$ denote the restriction of $\Pi$ to $B_i$. Since $\mathcal{G}\text{-}\mathrm{Stab}(\x)\subset\mathrm{Stab}(\x)$, we see that $F$ can be written as a linear combination of $\chi_{\Pi}$ and thus $F$ belongs to the kernel of $\mathscr{L}$.
\end{proof}
\subsection{Dyson vector flow analysis}
 The scheme of the analysis performed in this section is similar to \cites{cipolloni2022rank, cipolloni2022normal}. Throughout this section we analyze the flow \eqref{eq:maineqob} on the event $\Omega_1\cap\Omega_2$, with (cf. \cite{cipolloni2022rank}*{Eqs. (4.8), (4.29)}) 
 \begin{equation}
 \label{eq:defom1}
 \Omega_1=\Omega_{1,\xi}:=\Big\{ \sup_{0\le t \le T} \max_{i\in[N]} N^{2/3}\widehat{i}^{1/3} | \lambda_i(t)-\gamma_i(t)| \le N^\xi\Big\}
 \end{equation}
 where $\hat{i}=\min(i, N-i+1),$ and

\begin{equation}
	\label{eq:defom2}
	\begin{split}
		&\Omega_2=\Omega_{2,\xi,\omega}:=\bigcap_{\substack{\Re z_i\in [-3,3], \\ |\Im z_i|\ge N^{-1+\omega}}}\Bigg[ \left\{\sup_{0\le t \le T}\rho_{i,t}^{-1/2}\big|\langle G_t(z_1)\mathring{A}_1\rangle\big|\le \frac{N^\xi \rho_1}{N\sqrt{|\Im z_1|}}\langle \mathring{A}^2\rangle^{1/2}\right\}
		\\
		&\qquad\bigcap_{k=2}^n \left\{\sup_{0\le t \le T}(\rho_t^*)^{-1/2}\left|\langle G_t(z_1)\mathring{A}_1\dots G_t(z_k)\mathring{A}_k\rangle-\mathds{1}_{k=2}m_{t}(z_1)m_{t}(z_2)\langle \mathring{A}_1\mathring{A}_2\rangle\right|\le \frac{N^{\xi+k/2-1}}{\sqrt{N\eta_*}}\prod_{i=1}^k \langle \mathring{A}_i^2\rangle^{1/2}\right\} \Bigg],
	\end{split}
\end{equation}
with $G_t(z):=(W_t-z)^{-1}$, $m_t(z)$ denoting the Stieltjes transform of $\rho_t(x)=\sqrt{4(1+t^2)-x^2}/[2(1+t)\pi]$, $\eta_*:=\min_i |\Im z_i|$, $\rho_{i,t}:=|\Im m_t(z_i)|$, and $\rho_t^*:=\max_i \rho_{i,t}$. Here $\gamma_i(t)$ denote the \emph{quantiles} of $\rho_t(x)$, which are implicitly defined by
 \[
 \int_{-\infty}^{\gamma_i(t)}\rho_t(x)\, \mathrm{d}x=\frac{i}{N}.
 \]
The fact that $\Omega_1$ holds with very high probability follows e.g. by \cite{scgeneral}*{Theorem 7.6}, while from \cite{cipolloni2022rank}*{Theorem 2.2} it follows that $\Omega_2$ holds with very high probability. Note that on $\Omega_1\cap\Omega_2$ for $\langle A\rangle=0$ it holds (see \cite{cipolloni2022rank}*{Eqs. (4.30)--(4.31)})
\begin{equation}
\label{eq:necessb}
\big|\langle \u_i, A\u_j\rangle\big|^2\le \frac{N^{2\omega}\langle A^2\rangle}{N[\rho_t(\gamma_i(t)+\mathrm{i}N^{-2/3})\wedge \rho_t(\gamma_i(t)+\mathrm{i}N^{-2/3})]}.
\end{equation}

 We start by introducing an observable which is localized and is directed by a short-range approximation of the actual dynamics as introduced in \cites{bourgade2013eigenvector, marcinek2020high}.

\bed Let $\ell,\, K>$ be two $N$-dependent parameters, $\alpha\in(0,1)$ a given constant, and $\y$ a given configuration of $n$ particles such that $y_i\in\unn{\alpha N}{(1-\alpha)N}\eqqcolon \mathcal{I}_\alpha,$ we define the short-range observable by 
\begin{equation}\label{eq:shortrange}
		\left\{
\begin{array}{ll}
	\partial_t g_t(\x;\ell, K, \y) &= \mathscr{S}g_t(\x;\ell,K,\y)\\[1ex]
	g_0(\x;\ell,K,\y) &= \mathrm{Av}(\x;K,\y)\left(f_0(\x)-F(\x)\right)
\end{array}\right.
\end{equation}
where $\mathrm{Av}(\x;K,\y)f(\x)\coloneqq \frac{1}{K}\sum_{j=K}^{2K-1}f(\x)\mathds{1}_{\Vert \x-\y\Vert_1 \leqslant j}$ with $\Vert \x-\y\Vert_1=\sum_{i=1}^{2n}\vert x_i-y_i\vert$ and
\[
\mathscr{S}=\sum_{\substack{\vert i-j\vert \leqslant \ell\\ i\neq j \in \mathcal{I}_\alpha}}c_{ij}\mathscr{L}_{ij}.
\]
In the rest of the paper we may omit the dependence of $g_t$ on the parameters $\ell,\,K,$ and $\y$. 

\eed
We also define the following term
\beq\label{eq:renorm}
\Phi_{\mathbf{A}}=\prod_{i=1}^n\sqrt{\frac{\langle A_i^2\rangle}{N}}\quad\text{which gives}\quad \vert g_t(\x)\vert\leqslant N^{n\xi}\Phi_{\bf A} 
\eeq
with very high probability since $\mathscr{S}$ is a contraction in $L^\infty$ and $\vert g_0(\x)\vert\leqslant N^{n\xi}\Phi_{\bf A}$ on $\Omega_1\cap\Omega_2$ from \eqref{eq:necessb}.
The main goal of this section is to prove the following proposition.
\begin{prop}\label{prop:main}
	Let $\eta>0$, suppose that parameters are chosen so that $N^{-1}\ll\eta\ll T\ll \ell/N\ll K/N$. On the event $\Omega_1\cap\Omega_2$, it holds that 
\begin{equation}
\sum_{{\bf x}\in \Lambda^n}\pi({\bf x}) |g_T({\bf x})|^2\le K^n\mathcal{E}_n,
\quad\text{with}\quad \mathcal{E}_n=N^{n\xi}\Phi_{\bf A}
\left(
	\frac{N^\varepsilon \ell}{K} + \frac{NT}{\ell} + \frac{1}{\sqrt{K}}+ \frac{1}{\sqrt{N\eta}}
\right)
\end{equation}
with $\pi({\bf x})$ is the reversible measure for $\mathscr{L}$ and the bound holds with very high probability. 
\end{prop}

The start of the proof is similar to that of \cite{cipolloni2022normal}*{Proposition 4.5}. To see the main difference between the two proofs, we split the proof of Proposition \ref{prop:main} into the following lemma and proposition.
\bel\label{lemma:previouspaper}  On the event $\Omega_1\cap\Omega_2$, it holds that 
\begin{multline}
	\partial_t \Vert g_t\Vert_2^2
	\leqslant 
	-\frac{1}{\eta}\left(C_1(n)+\frac{1}{N\eta}\right)\Vert g_t\Vert_2^2+C(n)\frac{N^\xi K^{n-1}}{\eta}\Phi_{\bf A}
	+\frac{1}{\eta}\sum_{{\bf x}\in\Lambda^n}\sum_{{\bf i}}^*\sum_{{\bf a},{\bf b}}^*\mathrm{Av}(\x)\vert g_t(\x)\vert \Psi^{\bf i}_{{\bf a},{\bf b}}(\x)\times \\ \times
	\left(
		\sum_{{\bf j}}^*\left(\prod_{r=1}^{2n}a_{i_r,j_r}\right)\left(f_t({\bf x}_{{\bf a},{\bf b}}^{{\bf i}{\bf j}})-F({\bf x}_{{\bf a},{\bf b}}^{{\bf i}{\bf j}})\right)
		+\Phi_{\bf A}\O{\frac{N^{\varepsilon + n\xi}\ell}{K}+\frac{N^{1+n\xi}T}{\ell}+\frac{N^{1+n\xi}\eta}{\ell}}
	\right)
\end{multline}
\eel
\begin{proof}
By \cite{marcinek2020high}*{Lemma 4.16}, the operator $\mathscr{L}$ is negative in $L^2$, hence we get
\begin{equation}
\label{eq:l2bound}
\partial_t \lVert g_t\rVert_2^2=\sum_{|i-j|\le \ell} c_{ij}(t) \langle g_t, \mathscr{L}_{ij}(t) g_t\rangle\le \frac{1}{N\eta}\sum_{|j-i|\le \ell}\sum_{{\bf x}\in\Lambda^n}\frac{\eta}{(\lambda_i-\lambda_j)^2+\eta^2} g_t({\bf x})(\mathscr{M}_{ij}-\mathcal{E}_{ij})g_t({\bf x}).
\end{equation}
We first control the exchange operator $-\mathcal{E}_{ij},$
\begin{equation}
	\label{eq:errorE}
 \frac{1}{N\eta}\sum_{|j-i|\le \ell}\sum_{{\bf x}\in\Lambda^n}\frac{\eta}{(\lambda_i-\lambda_j)^2+\eta^2} g_t({\bf x})(-\mathcal{E}_{ij})g_t({\bf x})\le \frac{1}{N\eta^2}\lVert g_t\rVert_2^2,
\end{equation}
where we used that
\[
g_t({\bf x})\mathcal{E}_{ij}g_t({\bf x})\sim g_t({\bf x})[g(s_{ab}^{ij}{\bf x})-g({\bf x})]\le |g(s_{ab}^{ij}{\bf x})|^2+|g_t({\bf x})|^2,
\]
and that only finitely many (dependent on $n$) $i,j,a,b$ contributes to the summation for each ${\bf x}$.

As in \cite{cipolloni2022normal}, we replace the move operator $\mathcal{M}$ by
\[
\mathcal{A}:=\frac{1}{\eta}\sum_{{\bf i},{\bf j}\in [N]^{2n}}^*\left(\prod_{r=1}^{2n}a_{i_r,j_r}^\mathscr{S}\right)\sum_{{\bf a},{\bf b}}^*\big(g_t({\bf x}_{{\bf a},{\bf b}}^{{\bf i}{\bf j}})-g_t({\bf x})\big)
\quad\text{where}\quad
a_{i_r,j_r}:=\frac{\eta}{N[(\lambda_{i_r}-\lambda_{j_r})^2+\eta^2]},
\]
$a_{i_r,j_r}^\mathscr{S}=a_{i_r,j_r}$ for $|i_r-j_r|\le \ell$ and equal to zero otherwise, and
\[
{\bf x}_{{\bf a},{\bf b}}^{{\bf i}{\bf j}}={\bf x}+\left(\prod_{r=1}^{2n}\delta_{x_{a_r}i_r}\delta_{x_{b_r}i_r}\right)\sum_{r=1}^{2n}(j_r-i_r)({\bf e}_{a_r}+{\bf e}_{b_r}).
\]
Here $\sum^*$ denotes that the sum is over all distinct indices. By \cite{cipolloni2022normal}*{Lemma 4.6}, we have
\[
\frac{1}{N\eta}\sum_{|j-i|\le \ell}\sum_{{\bf x}\in\Lambda^n}\frac{\eta}{(\lambda_i-\lambda_j)^2+\eta^2} g_t({\bf x})\mathcal{M}_{ij}g_t({\bf x})
\le  C(n)\langle g_t, \mathcal{A} g_t\rangle\le 0,
\]
 which together with \eqref{eq:l2bound} and \eqref{eq:errorE}, implies
\begin{equation}
\begin{split}
\partial_t \lVert g_t\rVert_2^2&\le \frac{C(n)}{N\eta}\sum_{|j-i|\le \ell}\sum_{{\bf x}\in\Lambda^n}\frac{\eta}{(\lambda_i-\lambda_j)^2+\eta^2} g_t({\bf x})\mathcal{A}_{ij}g_t({\bf x})+\frac{1}{N\eta^2}\Vert g_t\Vert^2_2 \\
&=\frac{1}{\eta}\sum_{{\bf x}\in\Lambda^n}\sum_{{\bf i},{\bf j}\in [N]^{2n}}^*\left(\prod_{r=1}^{2n}a_{i_r,j_r}^\mathscr{S}\right)\sum_{{\bf a},{\bf b}}^*{g_t({\bf x})}\big(g_t({\bf x}_{{\bf a},{\bf b}}^{{\bf i}{\bf j}})-g_t({\bf x})\big)\Psi({\bf x})+\frac{1}{N\eta^2}\Vert g_t\Vert^2_2,
\end{split}
\end{equation}
where
\[
\Psi({\bf x}):=\prod_{r=1}^{2n}\delta_{x_{a_r}i_r}\delta_{x_{b_r}i_r}.
\]
For the term with $\vert g_t(\x)\vert^2 $ we can use the same bound as in \cite{cipolloni2022normal}*{(4.34)} and write
\begin{multline*}
-\sum_{\x \in \Lambda^n}
\vert g_t(\x)\vert^2
\sum_{{\bf a}, {\bf b}}
\sum_{{\bf i}, {\bf j}\in[N]^{2n}}^*
\Psi^{\bf i}_{{\bf a},{\bf b}}(\x)
\prod_{r=1}^{2n}a_{i_r,j_r}^\mathscr{S}\\
\leqslant 
-C(n)\sum_{\x\in\Lambda^n}
\vert g_t(\x)\vert^2
\sum_{{\bf a},{\bf b}}\sum_{\bf i}\Psi^{\bf i}_{{\bf a},{\bf b}}(\x)
\leqslant -C_1(n)\Vert g_t\Vert^2_2+C(n)N^{n\xi}\Phi_{\bf A}K^{n-1}
\end{multline*}
which gives that 
\begin{multline}
\partial_t \Vert g_t\Vert_2^2\leqslant -\frac{1}{\eta}\left(C_1(n)+\frac{1}{N\eta}\right)\Vert g_t\Vert_2^2+C(n)\frac{N^{n\xi} K^{n-1}}{\eta}\Phi_{\bf A}\\
+\frac{1}{\eta}\sum_{{\bf x}\in\Lambda^n}\sum_{{\bf i},{\bf j}\in [N]^{2n}}^*\left(\prod_{r=1}^{2n}a_{i_r,j_r}^\mathscr{S}\right)\sum_{{\bf a},{\bf b}}^*{g_t({\bf x})}g_t({\bf x}_{{\bf a},{\bf b}}^{{\bf i}{\bf j}})\Psi_{{\bf a},{\bf b}}^{\bf i}({\bf x})
\end{multline}

We can now use the fact that $F(\x)$ is in the kernel of $\mathscr{L}$ by Lemma \ref{lem:ansatz} and use similar bounds as \cite{marcinek2020high}*{(5.89-91), (5.95-97)} and \cite{cipolloni2022rank} to obtain
\begin{multline}
\sum_{j}^*
\left(\prod_{i=1}^{2n}
	a_{i_r,j_r}^{\mathscr{S}}
\right)
\sum_{{\bf a},{\bf b}}^*g_t(\x^{{\bf i}{\bf j}}_{{\bf a},{\bf b}})
=
\mathrm{Av}(\x)\sum_{j}^*
\left(\prod_{r=1}^{2n}
a_{i_r,j_r}
\right)
\sum_{{\bf a},{\bf b}}^*\left(
	f_t(\x^{{\bf i}{\bf j}}_{{\bf a},{\bf b}})-F({\bf x}_{{\bf a},{\bf b}}^{{\bf i}{\bf j}})
\right)\\
+\O{\frac{N^{\varepsilon + n\xi}\ell}{K}\Phi_{\bf A}+\frac{N^{1+n\xi}T}{\ell}\Phi_A+\frac{N^{1+n\xi}\eta}{\ell}\Phi_{\bf A}}
\end{multline}
which gives the results.
\end{proof}
We now need to analyze the term involving the configurations where jumps have occurred. The main difference with previous analysis from \cite{cipolloni2022normal}
 is that we do not have a sum over all perfect matchings of the configurations with the same matrix $A$ but we need to consider each colored perfect matching separately.
 \bel  On the event $\Omega_1\cap\Omega_2$, we have that
\[
\sum_{{\bf i}}^*\sum_{{\bf a},{\bf b}}\Psi^{\bf i}_{{\bf a},{\bf b}}(\x)\sum_{{\bf j}}^*\left(\prod_{r=1}^{2n}a_{i_r,j_r}\right)\left(f_t({\bf x}_{{\bf a},{\bf b}}^{{\bf i}{\bf j}})-F({\bf x}_{{\bf a},{\bf b}}^{{\bf i}{\bf j}})\right)
=
\O{\frac{N^{n\varepsilon}}{\sqrt{N\eta}}\Phi_{\bf A}}.
\]
\eel
\begin{proof}
We start by using the graph convention from \eqref{eq:graph}. Remember that we are going to make particles jump to sites where there are no particles, thus two vertices from each site are going to jump to a new empty site. 
There are several possible cases for our configuration after jumps ${\bf x}_{{\bf a},{\bf b}}^{{\bf i}{\bf j}},$ we now fix an edge $e\in\mathcal{E}$ and consider the different cases both for our observable $f_t$ and for the ansatz observable $F$. For simplicity, we introduce the following notation 
\[
F_e(\mathcal{G})=\sum_{\Pi\in \mathrm{PM}[n(e)]}\sum_{\{i,j\}\in \Pi}\langle A_{e^{(i)}}A_{e^{\Pi(i)}}\rangle
\]
Note that since we ${\bf x}_{{\bf a},{\bf b}}^{{\bf i}{\bf j}}$ has either 2 or 0 particles at each site, we see that $(n(i)-1)!!=1$ and we can write
\[
F({\bf x}_{{\bf a},{\bf b}}^{{\bf i}{\bf j}})=\frac{1}{N^{n/2}}\prod_{e=\{k,\ell\}\in\tilde{\mathcal{E}}}2^{n(e)\mathds{1}_{k=\ell}}F_e(\mathcal{G_{{\bf a},{\bf b}}^{{\bf i}{\bf j}}})
\]
\paragraph{Case 1: $e=\{i,i\}$.} For this case, we have three possible cases of jumps. We recall that at the site $i$, we have $n(e)$ occurrences of the edge $\{i,i\}$.
\begin{itemize}
	\item If for at least one occurrence of the edge, both vertices are selected to move to the same new site $j_r$, we obtain a subleading term. Indeed, the resulting graph has only one occurrence of the edge $e'=\{j_r,j_r\}$ (see Figure~\ref{fig:jumpsame}). Thus, since it is impossible to perform a perfect matching on one vertex, we have $F_{e'}(\mathcal{G}_{{\bf a},{\bf b}}^{{\bf i}{\bf j}})=0$ and thus $F(\x_{{\bf a},{\bf b}}^{{\bf i}{\bf j}})=0$. For the observable $f_t(\x_{{\bf a},{\bf b}}^{{\bf i}{\bf j}}),$ it means that we obtain a term of the form
	\begin{equation*}
	\sum_{j_r=1}^*a_{i,j_r}\scp{\u_{j_r}}{A\u_{j_r}} = \sum_{j_r=1}^* \frac{\eta \scp{\u_{j_r}}{A\u_{j_r}}}{N[(\lambda_{i_r}-\lambda_{j_r})^2+\eta^2]}
	=
	\Im \langle G_t(z_{i})A\rangle + \O{\frac{N^{\xi}\sqrt{\langle A^2\rangle}}{N^{3/2}\eta}}
	=
	\O{\frac{N^{\xi}\sqrt{\langle A^2\rangle}}{N\sqrt{\eta}}}
	\end{equation*}
where we added in the second equality the finitely many missing terms in the sum and we used the matrix-valued local law \eqref{eq:defom2} in the third one. So that we can bound,
 for this case of jumps,
\[
\sum_{\bf j}^*\left(\prod_{r=1}^{2n}a_{i_r,j_r}\right)
\left(
f_t(\x_{{\bf a},{\bf b}}^{{\bf i}{\bf j}})-F(\x_{{\bf a},{\bf b}}^{{\bf i}{\bf j}})
\right)
=
\O{\frac{N^{n\xi}}{N\sqrt{\eta}}\sqrt{\langle A_{i_r}^2\rangle}\prod_{i\neq i_r}\sqrt{\frac{\langle A_{i}\rangle }{N}}}=\O{\frac{N^{n\xi}}{\sqrt{N\eta}}\Phi_{\bf A}}.
\]
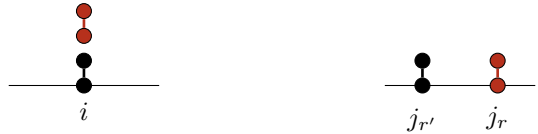
\begin{figure}[!ht]
	\centering 
	\begin{tikzpicture}
		\draw[black] (-1,0) -- (1,0);
		
		\node[draw,black,fill=black, label=below:$\displaystyle{i}$] (1) at (0,0) {};
		\node[draw,black,fill=black,] (2) at (0,.33) {};
		\node[draw,black, fill=BrickRed] (3) at (0,.66) {};
		\node[draw,black, fill=BrickRed] (4) at (0,.99) {};
		
		 \path[-, black,line width=.1em] (1) edge (2);
		 \path[-, BrickRed,line width=.1em] (3) edge (4);

		\begin{scope}[xshift = 5cm]
			\draw[black] (-1,0) -- (1,0);
			
			\node[draw,black,fill=black, label=below:$\displaystyle{j_{r'}}$] (11) at (-0.5,0) {};
			\node[draw,black,fill=black,] (12) at (-0.5,.33) {};
			\node[draw,black, fill=BrickRed, label=below:$\displaystyle{j_{r}}$] (13) at (0.5,.0) {};
			\node[draw,black, fill=BrickRed] (14) at (0.5,.33) {};
			
			\path[-, black,line width=.1em] (11) edge (12);
			\path[-, BrickRed,line width=.1em] (13) edge (14);
		\end{scope}
	\end{tikzpicture}
\caption{This jump structure makes each edge from the initial configuration jump together to an edge on a distinct sites.}
\label{fig:jumpsame}
\end{figure}
\item We can take vertices from a cycle of occurrences $e^{(m_1)}$,\dots, $e^{(m_p)}.$ In the sense that each pair of vertices which jumps to the same site $j_r$ is taken from two different occurrences. Suppose that $k\geqslant 3$, this creates $p$ single edges and thus for the same reason as above $F(\x_{\bf{a},\bf{b}}^{\bf{i}\bf{j}})=0$. For $f_t(\x_{{\bf a},{\bf b}}^{{\bf i}{\bf j}})$, we obtain a term of the form,
\begin{multline*}
\sum_{j_1,\dots j_p=1}^Na_{i,j_1}\dots a_{i,j_p}\scp{\u_{j_1}}{A'_1\u_{j_2}}\scp{\u_{j_2}}{A'_2\u_{j_3}}\dots\scp{\u_{j_p}}{A'_p\u_{j_1}}
=
N^{1-p}\langle \Im G_t(z_i)A'_1\dots \Im G_t(z_i)A'_p\rangle\\
=
\O{\frac{N^{\xi-p/2}}{\sqrt{N\eta}}\prod_{i=1}^p\sqrt{\langle (A'_i)^2\rangle}}.
\end{multline*}
where we reindexed the matrices as $A'_i=A_{e^{(m_i)}}.$
Thus, by trivially bounding the rest of the edges as for above, we obtain a term 
\[
\sum_{\bf j}^*\left(\prod_{r=1}^{2n}a_{i_r,j_r}\right)
\left(
f_t(\x_{{\bf a},{\bf b}}^{{\bf i}{\bf j}})-F(\x_{{\bf a},{\bf b}}^{{\bf i}{\bf j}})
\right)
=
\O{\frac{N^{(n-p+1)\xi}}{\sqrt{N\eta}}\Phi_{\bf A}}
\]
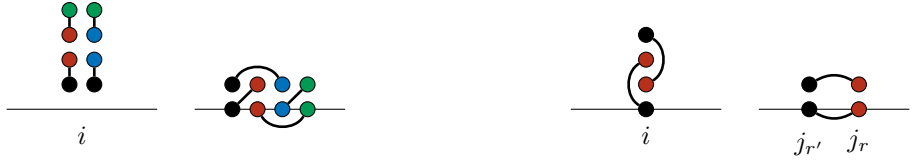
\begin{figure}[!ht]
	\centering
	\begin{tikzpicture}
		\draw[black] (-1,0) -- (1,0);
		\node[draw=none,black,label=below:\textcolor{black}{$i$}] at (0,0) {};
		
		\node[draw,black,fill=black] (1) at (-0.165,0.33) {};
		\node[draw,black,fill=BrickRed] (2) at (-0.165,0.66) {};
		\node[draw,black, fill=BrickRed] (3) at (-0.165,0.99) {};
		\node[draw,black, fill=ForestGreen] (4) at (-0.165,1.32) {};
		
		\node[draw,black,fill=black] (11) at (0.165,0.33) {};
		\node[draw,black,fill=RoyalBlue] (12) at (0.165,0.66) {};
		\node[draw,black, fill=RoyalBlue] (13) at (0.165,0.99) {};
		\node[draw,black, fill=ForestGreen] (14) at (0.165,1.32) {};
		
		\path[-, black,line width=.1em] (11) edge (12);
		\path[-, black,line width=.1em] (1) edge (2);
		\path[-, black,line width=.1em] (13) edge (14);
		\path[-, black,line width=.1em] (3) edge (4);

		\begin{scope}[xshift=2.5cm]
			\draw[black] (-1,0) -- (1,0);
			\node[draw,black,fill=black] (1) at (-0.5,0.33) {};
			\node[draw,black,fill=black] (2) at (-0.5,0) {};
			\node[draw,black, fill=BrickRed] (3) at (-0.165,0) {};
			\node[draw,black, fill=BrickRed] (4) at (-0.165,.33) {};
			
			\node[draw,black,fill=RoyalBlue] (11) at (0.165,0.33) {};
			\node[draw,black,fill=RoyalBlue] (12) at (0.165,0) {};
			\node[draw,black, fill=ForestGreen] (13) at (0.5,0) {};
			\node[draw,black, fill=ForestGreen] (14) at (0.5,.33) {};
			
			\path[-, black,line width=.1em] (2) edge (4);
			\path[-, black,line width=.1em] (3) edge[bend right=60] (13);
			\path[-, black,line width=.1em] (12) edge (14);
			\path[-, black,line width=.1em] (11) edge[bend right=60] (1);
		\end{scope}
		\begin{scope}[xshift=7.5cm]
			\draw[black] (-1,0) -- (1,0);
			
			\node[draw,black, fill=black, label=below:$\textcolor{black}{i}$] (1) at (0,0) {};
			\node[draw,black, fill=BrickRed] (2) at (0,.33) {};
			\node[draw,black, fill=BrickRed] (3) at (0,.66) {};
			\node[draw,black, fill=black] (4) at (0,.99) {};
			
			\path[-,black,line width=.1em] (1) edge[bend left=60] (3);
			\path[-,black,line width=.1em] (4) edge[bend left=60] (2);
			
			\begin{scope}[xshift = 2.5cm]
				\draw[black] (-1,0) -- (1,0);
				
				\node[draw,black,fill=black, label=below:$\textcolor{black}{j_{r'}}$] (11) at (-0.33,0) {};
				\node[draw,black,fill=black,] (12) at (-0.33,.33) {};
				\node[draw,black, fill=BrickRed, label=below:$\textcolor{black}{j_{r}}$] (13) at (0.33,.0) {};
				\node[draw,black, fill=BrickRed] (14) at (0.33,.33) {};
				
				\path[-, black,line width=.1em] (11) edge[bend right=30] (13);
				\path[-, black,line width=.1em] (12) edge[bend left=30] (14);
			\end{scope}
			
		\end{scope}
	\end{tikzpicture}
\caption{These jumps take two vertices from different edges and place them on the same site $j_r$. We see that a cycle has been created in the new configuration of particles. If $k\geqslant 3$, as in the left figure ($k=4$), the ansatz observable vanishes while if $k=2$, as in the right, the ansatz observable is nonzero.}
\label{fig:jumpdif}
\end{figure}
\item Now, if we consider two occurrences of the edge $e=\{x_{2k-1},x_{2k}\}$ and $e'=\{x_{2\ell-1},x_{2\ell}\}$ such that $x_{2k-1}=x_{2k}=x_{2\ell-1}=x_{2\ell}=i.$ If we choose the two particles $x_{2k-1}$ and $x_{2\ell}$ to jump to $j_r$ and $x_{2k}$ and $x_{2\ell-1}$ to jump to $j_{r'}$, the new graph has exactly two occurrences of the edge $e''=\{j_r,j_{r'}\}$ (see Figure~\ref{fig:jumpdif}). Thus, $F_{e''}(\mathcal{G}_{{\bf a},{\bf b}}^{{\bf i}{\bf j}})=\langle A_{k}A_{\ell}\rangle$ and 
\[
\sum_{j_r,j_{r'}=1}^Na_{i,j_r}a_{i,j_{r'}}F_{e''}(\mathcal{G}_{{\bf a},{\bf b}}^{{\bf i}{\bf j}})={(\Im m_t(z_i))^2}\langle A_{k}A_{\ell}\rangle. 
\] 
For the observable, the same reasoning holds and we obtain
\[
\sum_{j_r,j_{r'}=1}^Na_{i,j_r}a_{i,j_{r'}}\scp{\u_{j_r}}{A_k\u_{j_{r'}}}\scp{\u_{j_r}}{A_\ell\u_{j_{r'}}}
=
{(\Im m_t(z_i))^2}\langle A_{k}A_{\ell}\rangle
+
\O{\frac{N^\xi\sqrt{\langle A_k^2\rangle\langle A_\ell^2\rangle}}{\sqrt{N\eta}}}.
\]
\end{itemize}

\paragraph{Case 2: $e=\{i_1,i_2\}$ with $i_1\neq i_2$.} In this case we have two subcases when choosing the jumping particles.
\begin{itemize}
	\item Suppose that the particle at site $i_1$ and $i_2$ jump respectively to the sites $j_1$ and $j_2$. If at the site $i_1$, the second chosen particle is attached to a
	particle at a site $i_p$ with $i_p\neq i_2$. Then after the jumps, the edge $\{j_{1},j_2\}$ will be a single edge and thus $F(\x_{{\bf a}{\bf b}}^{{\bf i}{\bf j}})=0.$ For the observable
	$f_t$, this creates a cycle of a certain order as in the second point of the Case 1 above and we get that, for some $p>0$,
	\[
		\sum_{\bf j}^*\left(\prod_{r=1}^{2n}a_{i_r,j_r}\right)
		\left(
		f_t(\x_{{\bf a},{\bf b}}^{{\bf i}{\bf j}})-F(\x_{{\bf a},{\bf b}}^{{\bf i}{\bf j}})
		\right)
		=
		\O{\frac{N^{(n-p+1)\xi}}{\sqrt{N\eta}}\Phi_{\bf A}}.	
	\]
	\item Suppose again that the particle at site $i_1$ and $i_2$ jump respectively to the sites $j_1$ and $j_2.$ Suppose now that the other chosen particle at site $i_1$ belongs to an occurrence of the edge $\{i_1,i_2\}$ and that its attached vertex is also jumping to site $j_2,$ then we are exactly in the same case as the third item of Case 1.
	
\end{itemize}
Thus, we see that if we are in the first or second item of Case 1 or the first item of Case 2 for at least one edge, we obtain that 
\[
\sum_{\bf j}^*\left(\prod_{r=1}^{2n}a_{i_r,j_r}\right)
\left(
f_t(\x_{{\bf a},{\bf b}}^{{\bf i}{\bf j}})-F(\x_{{\bf a},{\bf b}}^{{\bf i}{\bf j}})
\right)
=
\O{\frac{N^{n\xi}}{\sqrt{N\eta}}\Phi_{\bf A}}.
\]

If we perform for all particles in the configuration, jumps as in the third item of Case 1 or the second of Case 2 (if possible), we obtain an error of 
\[
\sum_{\bf j}^*\left(\prod_{r=1}^{2n}a_{i_r,j_r}\right)
\left(
f_t(\x_{{\bf a},{\bf b}}^{{\bf i}{\bf j}})-F(\x_{{\bf a},{\bf b}}^{{\bf i}{\bf j}})
\right)
=
\O{\frac{N^{n\xi}}{\sqrt{N\eta}}\sqrt{\prod_{k=1}^n\frac{\langle A_k^2\rangle}{N}}}
=
\O{\frac{N^{n\xi}}{\sqrt{N\eta}}\Phi_{\bf A}}
\]
which concludes the proof.
\end{proof}
We can now proof Proposition \ref{prop:main}.
\begin{proof}[Proof of Proposition \ref{prop:main}] The proof is now the same as the $L^2$ bound proved in \cite{cipolloni2022normal}*{Equation (4.50) and beyond} where the error $\mathcal{E}$ is slightly changed to take in accound the different matrices from $\mathbf{A}$.
\end{proof}

The $L^2$ bound from Proposition \ref{prop:main} can now be converted into a $L^\infty$ bound using the same techniques as \cite{cipolloni2022normal}*{Section 4.4} based on the estimates from \cite{marcinek2020high}. The proof is the same since we need the analysis of the \emph{colored} eigenvector moment flow which has been done in the two papers above (even though \cite{cipolloni2022normal} did not needed it \emph{per se}). For this reason, we give the following corollary without proof and refer to \cite{cipolloni2022normal}*{Section 4.4} for more details.

\begin{cor}
\label{cor:fincor}
	Let $\delta>0$. For any $n\in\N$, there exists a constant $c(n)$ such that for any $\varepsilon>0$ and $T\geqslant N^{-1+\varepsilon}$, we have
	\[
	\frac{1}{\Phi_{\mathbf{A}}}\sup_{\mathcal{G}}\vert f_{\bm{\lambda},T}(\mathcal{G},\mathbf{A})-F(\mathcal{G},\mathbf{A})\vert \leqslant C(n,\varepsilon,\delta)N^{-c(n)}
	\]
	where the supremum holds for $\mathcal{G}=(\mathcal{V},\mathcal{E})$ with $\vert \mathcal{V}\vert =2n$ and $\mathcal{V}\subset \mathcal{I}_\alpha^{2n}$and the bounds hold with very high probability.
\end{cor}	

We conclude this section with the proof of Theorem~\ref{theo:main}.

\begin{proof}[Proof of Theorem~\ref{theo:main}]

Corollary~\ref{cor:fincor} proves Theorem~\ref{theo:main} for matrices $W$ with a small Gaussian component of size $T\sim N^{-1+\varepsilon}$. Then, this Gaussian component can be easily removed by a standard Green's Function Comparison (GFT) argument (see Appendix~\ref{sec:GFT} for a detailed argument).

\end{proof}

\appendix

\section{Green's Function Comparison (GFT)}
\label{sec:GFT}

Consider the Ornstein--Uhlenbeck flow
\begin{equation}
\label{eq:OU}
\d \hat{W}_t=-\frac{\hat{W}_t}{2}\d t+\frac{\d \hat{B}_t}{\sqrt{N}},\qquad \hat{W}_0=W,
\end{equation}
with $\hat{B}_t$ being a matrix valued real symmetric or complex Hermitian Brownian motion, and $W$ being a Wigner matrix satisfying Assumption~\ref{ass:momass}. Note that along the flow \eqref{eq:OU} the first two moments of $\hat{W}_t$ are preserved. Note also that the dynamics is slightly different from \eqref{eq:DBM} but it is chosen as an initial condition for \eqref{eq:DBM} such that at time $T$, their distributions are the same, see \cite{cipolloni2022rank}*{(B.1)--(B.3)} for more details. We denote the resolvent of $\hat{W}_t$ by $G_t(z):=(\hat{W}_t-z)^{-1}$, with $z\in \C\setminus\R$. 

In this section we show that for $t$ sufficiently small the distribution of the eigenvectors of $\hat{W}_t$ is very close to their distribution at time $t=0$. The next proposition shows that as long as $t\ll N^{1/2}$ the change in the distribution of eigenvectors is negligible. We remark that for the purpose of proving Theorem~\ref{theo:main} we need to choose $t=N^{-1+\epsilon}$, for some small fixed $\epsilon>0$, hence much smaller that the threshold $N^{-1/2}$. The fact that to prove Theorem~\ref{theo:main}  it is enough to prove Proposition~\ref{pro:GFT} for resolvents at a spectral parameter $\eta=N^{-1-\zeta}$ follows by level repulsion \cite{knowles2013eigenvector}.

\begin{prop}
\label{pro:GFT}
Fix any $p\in\N,\zeta, \delta>0$,and consider a smooth function $F:\R^p\to\R$. Then there exists a constant $C>0$ such that
\begin{equation}
\label{eq:borneed}
\left|\E F\left(N^{-(k-2)/2}C_k^{-1} \langle \Im G_t(z_1) A_1\dots \Im G_t(z_k)A_k\rangle   \right)_{k\in [p]}-\E F\left(\Im G_t \to \Im G_0\right)_{k\in [p]}\right|\lesssim t N^{1/2+C\zeta}.
\end{equation}
Here $z_i=E_i+\I\eta$, with $\eta=N^{-1-\zeta}$ and $E_i \in [-2+\delta,2-\delta]$, $A_1,A_2,\dots$ are Hermitian matrices with $\langle A_i\rangle=0$, and
\[
C_k=C_k(A_1,\dots,A_k):= \prod_{i\in [k]} \langle A_i^2\rangle^{1/2}.
\]
\end{prop}
\begin{proof}

The proof of this proposition is similar to \cite{cipolloni2022rank}*{Appendix B}, we thus only present an outline of it. To keep the notation short we define
\[
R_t:= F\left(N^{-(k-2)/2}C_k^{-1} \langle \Im G_t(z_1) A_1\dots \Im G_t(z_k)A_k\rangle   \right)_{k\in [p]}.
\]
Then, by It\^{o}'s formula, we compute
\begin{equation}
\label{eq:needexp}
\E\frac{\d}{\d t} R_t=\E\left[-\frac{1}{2} \sum_\alpha w_\alpha(t) \partial_\alpha R_t+\frac{1}{2}\sum_{\alpha,\beta} \kappa_t(\alpha,\beta) \partial_\alpha \partial_\beta R_t \right],
\end{equation}
where $\alpha,\beta\in [N]^2$ are double indices, $w_\alpha(t)$ are the entries of $\hat{W}_t$, and $\partial_\alpha:=\partial_{w_\alpha}$. Here
\[
\kappa_t(\alpha_1,\dots,\alpha_l):=\kappa(w_{\alpha_1}(t), \dots,w_{\alpha_l}(t))
\]
denotes the joint cumulant of $w_{\alpha_1}(t), \dots,w_{\alpha_l}(t)$, with $l\in\N$. Note that by \eqref{eq:momass} it follows $|\kappa_t(\alpha_1,\dots,\alpha_l)|\lesssim N^{-l/2}$ uniformly in $t\ge 0$. Performing a cumulant expansion in \eqref{eq:needexp} we are left with
\begin{equation}
\E\frac{\d}{\d t} R_t=\E\sum_{l=3}^L \sum_{\alpha_1,\dots,\alpha_l}\kappa_t(\alpha_1,\dots,\alpha_l) \partial_{\alpha_1}\cdots \partial_{\alpha_l}R_t+\Omega(L),
\end{equation}
with $\Omega(L)$ a negligible error.

Next, we note that by \cite{cipolloni2022rank}*{Corollary 2.4} it follows that
\begin{equation}
\label{eq:iso}
\big|\langle {\bm x}, G_t(z_1) A_1\dots  G_t(z_k)A_k G_t(z_{k+1}) {\bm y}\rangle \big| \lesssim N^{C\zeta} N^{k/2}\prod_{i\in [k]} \langle A_i^2\rangle^{1/2},
\end{equation}
for some constant $C>0$ possibly depending on $k$. We point out that in \cite{cipolloni2022rank}*{Corollary 2.4}  the bound \eqref{eq:iso} is proven for $\Im z_i\ge N^{-1+\epsilon}$, for any small $\epsilon>0$; the fact that it can be extended to $\Im z_i=\eta$, i.e. a bit below the level spacing $N^{-1}$, follows similarly to \cite{cipolloni2021edge}*{Appendix A}. We now show this considering a specific example, the general case is completely analogous and so omitted; we explain the details at the end of this section.

Finally, by \eqref{eq:iso} and $|(G)_{ab}|\le N^\zeta$ it readily follows that for any $k\in [p]$ we have
\begin{equation}
\left|\partial_{\alpha_1}\cdots \partial_{\alpha_l}N^{-(k-2)/2}C_k^{-1} \langle \Im G_t(z_1) A_1\dots \Im G_t(z_k)A_k\rangle\right|\le N^{C\zeta},
\end{equation}
for some large fixed constant $C>0$. We thus conclude that
\begin{equation}
\left|\E\frac{\d}{\d t} R_t\right|\lesssim \sum_{l=3}^L N^{-l/2+2} N^{C\zeta}= N^{1/2+C\zeta}.
\end{equation}
Integrating in time we conclude \eqref{eq:borneed}.

\end{proof}

\begin{proof}[Proof of \eqref{eq:iso}]

For simplicity we do not distinguish the indices, i.e. we write
\[
(GA)^l=G_t(z_1) A_1\dots  G_t(z_l)A_l.
\]
 Assume that $\Im z_l=\eta$ and $|\Im z_j|\ge\eta_*:= N^{-1+\epsilon}$ for $j\ne l$, then 
\begin{equation}
\begin{split}
\langle {\bm x}, (GA)^{l-1}G(z_l)A(GA)^{k-l} G{\bm y}\rangle &=\langle {\bm x}, (GA)^{l-1}G(E_l+\I \eta_*)A(GA)^{k-l} G {\bm y}\rangle \\
&\quad-\int_\eta^{\eta_*}\langle {\bm x}, (GA)^{l-1}G(E_{q_0}+\I \tau)^2A(GA)^{k-l}G {\bm y}\rangle\,\d\tau.
\end{split}
\end{equation}
For the first term in the rhs. we use the local law from  \cite{cipolloni2022rank}*{Corollary 2.4} obtaining a bound as in \eqref{eq:iso}. For the second term we estimate (neglecting $\log N$--factors)
\begin{equation}
\begin{split}
&\int_\eta^{\eta_*}\langle {\bm x}, (GA)^{l-1}G(E_l+\I \tau)^2A(GA)^{k-l}G {\bm y}\rangle\,\d\tau \\
&\lesssim \int_\eta^{\eta_*}\frac{1}{\tau}\langle {\bm x}, (GA)^{l-1}\Im G(E_l+\I \tau)((GA)^{l-1})^*{\bm x}\rangle^{1/2}  \langle {\bm y}, (A(GA)^{k-l}G)^*\Im G(E_l+\I \tau)A(GA)^{k-l}G{\bm y}\rangle^{1/2} \, \d\tau\\
&\lesssim  \frac{\eta_*}{\eta} \int_\eta^{\eta_*}\frac{1}{\tau}\langle {\bm x}, (GA)^{l-1}\Im G(E_l+\I \eta_*)((GA)^{l-1})^*{\bm x}\rangle^{1/2}\langle {\bm y}, (A(GA)^{k-l}G)^*\Im G(E_l+\I \eta_*)A(GA)^{k-l}G{\bm y}\rangle^{1/2} \, \d\tau\\
&\lesssim \frac{\eta_*}{\eta} \sqrt{N^{l-1} N^{k-l+1}\prod_{i\in [k]} \langle A_i^2\rangle}\lesssim N^{\zeta} N^{k/2}\prod_{i\in [k]} \langle A_i^2\rangle^{1/2}.
\end{split}
\end{equation}
This proves \eqref{eq:iso} choosing $\epsilon\le \zeta$. We remark that in the second inequality we used the monotonicity of the map $\tau\mapsto \tau \langle {\bm v}, \Im G(E+\I\tau) {\bm w}\rangle$, and that in the third inequality we used the local law from  \cite{cipolloni2022rank}*{Corollary 2.4}.

\end{proof}

\bibliography{mass}

\begin{bibdiv}
\begin{biblist}

\bib{abert2018eigenfunctions}{article}{
      author={Abert, M.},
      author={Bergeron, N.},
      author={Le~Masson, E.},
       title={Eigenfunctions and random waves in the {B}enjamini--{S}chramm
  limit},
        date={2018},
     journal={arXiv preprint arXiv:1810.05601},
}

\bib{anantharaman2019quantum}{article}{
      author={Anantharaman, N.},
      author={Sabri, M.},
       title={Quantum ergodicity on graphs: from spectral to spatial
  delocalization},
        date={2019},
     journal={Ann. of Math. (2)},
      volume={189},
      number={3},
       pages={753\ndash 835},
}

\bib{benigni2021fermionic}{article}{
      author={Benigni, L.},
       title={Fermionic eigenvector moment flow},
        date={2021},
        ISSN={0178-8051},
     journal={Probab. Theory Related Fields},
      volume={179},
      number={3-4},
       pages={733\ndash 775},
}

\bib{berry1977regular}{article}{
      author={Berry, M.~V.},
       title={Regular and irregular semiclassical wavefunctions},
        date={1977},
     journal={Journal of Physics A: Mathematical and General},
      volume={10},
      number={12},
       pages={2083},
}

\bib{bohigas1984characterization}{article}{
      author={Bohigas, O.},
      author={Giannoni, M.-J.},
      author={Schmit, C.},
       title={Characterization of chaotic quantum spectra and universality of
  level fluctuation laws},
        date={1984},
     journal={Physical review letters},
      volume={52},
      number={1},
       pages={1},
}

\bib{benigni2022fluctuations}{article}{
      author={Benigni, L.},
      author={Lopatto, P.},
       title={Fluctuations in local quantum unique ergodicity for generalized
  {W}igner matrices},
        date={2022},
        ISSN={0010-3616},
     journal={Comm. Math. Phys.},
      volume={391},
      number={2},
       pages={401\ndash 454},
         url={https://doi.org/10.1007/s00220-022-04314-z},
      review={\MR{4397177}},
}

\bib{benigni2022optimal}{article}{
      author={Benigni, L.},
      author={Lopatto, P.},
       title={Optimal delocalization for generalized wigner matrices},
        date={2022},
     journal={Advances in Mathematics},
      volume={396},
       pages={108109},
}

\bib{bru1989diffusions}{article}{
      author={Bru, M.-F.},
       title={Diffusions of perturbed principal component analysis},
        date={1989},
     journal={Journal of multivariate analysis},
      volume={29},
      number={1},
       pages={127\ndash 136},
}

\bib{bourgade2013eigenvector}{article}{
      author={Bourgade, P.},
      author={Yau, H.-T.},
       title={The eigenvector moment flow and local quantum unique ergodicity},
        date={2017},
        ISSN={0010-3616},
     journal={Comm. Math. Phys.},
      volume={350},
      number={1},
       pages={231\ndash 278},
}

\bib{bourgade2018random}{article}{
      author={Bourgade, P.},
      author={Yau, H.-T.},
      author={Yin, J.},
       title={Random band matrices in the delocalized phase {I}: Quantum unique
  ergodicity and universality},
        date={2020},
     journal={Comm. Pure Appl. Math.},
      volume={73},
      number={7},
       pages={1526\ndash 1596},
}

\bib{de1985ergodicite}{article}{
      author={Colin De~Verdière, Y.},
       title={Ergodicit{\'e} et fonctions propres du laplacien},
        date={1985},
     journal={Comm. Math. Phys.},
      volume={102},
      number={3},
       pages={497\ndash 502},
}

\bib{cipolloni2021edge}{article}{
      author={Cipolloni, G.},
      author={Erd\H{o}s, L.},
      author={Schr\"{o}der, D.},
       title={Edge universality for non-hermitian random matrices},
        date={2021},
     journal={Probability Theory and Related Fields},
      volume={179},
      number={1},
       pages={1\ndash 28},
}

\bib{cipolloni2021eigenstate}{article}{
      author={Cipolloni, G.},
      author={Erd\H{o}s, L.},
      author={Schr\"{o}der, D.},
       title={Eigenstate thermalization hypothesis for {W}igner matrices},
        date={2021},
        ISSN={0010-3616},
     journal={Comm. Math. Phys.},
      volume={388},
      number={2},
       pages={1005\ndash 1048},
         url={https://doi.org/10.1007/s00220-021-04239-z},
      review={\MR{4334253}},
}

\bib{cipolloni2022normal}{article}{
      author={Cipolloni, G.},
      author={Erd\H{o}s, L.},
      author={Schr\"{o}der, D.},
       title={Normal fluctuation in quantum ergodicity for {W}igner matrices},
        date={2022},
        ISSN={0091-1798},
     journal={Ann. Probab.},
      volume={50},
      number={3},
       pages={984\ndash 1012},
         url={https://doi.org/10.1214/21-aop1552},
      review={\MR{4413210}},
}

\bib{cipolloni2022optimal}{article}{
      author={Cipolloni, G.},
      author={Erd\H{o}s, L.},
      author={Schr\"{o}der, D.},
       title={Optimal multi-resolvent local laws for {W}igner matrices},
        date={2022},
     journal={Electronic Journal of Probability},
      volume={27},
       pages={1\ndash 38},
}

\bib{cipolloni2022rank}{article}{
      author={Cipolloni, G.},
      author={Erd\H{o}s, L.},
      author={Schr\"{o}der, D.},
       title={Rank-uniform local law for {W}igner matrices},
        date={2022},
     journal={Forum Math. Sigma},
      volume={10},
       pages={Paper No. e96},
         url={https://doi.org/10.1017/fms.2022.86},
      review={\MR{4502022}},
}

\bib{deutsch2018eigenstate}{article}{
      author={Deutsch, J.},
       title={Eigenstate thermalization hypothesis},
        date={2018},
     journal={Rep. Prog. Phys.},
      volume={81},
      number={8},
       pages={082001},
}

\bib{scgeneral}{article}{
      author={Erd\H{o}s, L.},
      author={Knowles, A.},
      author={Yau, H.-T.},
      author={Yin, J.},
       title={The local semicircle law for a general class of random matrices},
        date={2013},
     journal={Electronic Journal of Probability},
      volume={18},
      number={59},
}

\bib{erdos2009local}{article}{
      author={Erd\H{o}s, L.},
      author={Schlein, B.},
      author={Yau, H.-T.},
       title={Local semicircle law and complete delocalization for {W}igner
  random matrices},
        date={2009},
        ISSN={0010-3616},
     journal={Comm. Math. Phys.},
      volume={287},
      number={2},
       pages={641\ndash 655},
}

\bib{erdos2009semicircle}{article}{
      author={Erd\H{o}s, L.},
      author={Schlein, B.},
      author={Yau, H.-T.},
       title={Semicircle law on short scales and delocalization of eigenvectors
  for {W}igner random matrices},
        date={2009},
        ISSN={0091-1798},
     journal={Ann. Probab.},
      volume={37},
      number={3},
       pages={815\ndash 852},
}

\bib{erdos2011universality}{article}{
      author={Erd\H{o}s, L.},
      author={Schlein, B.},
      author={Yau, H.-T.},
       title={Universality of random matrices and local relaxation flow},
        date={2011},
     journal={Invent. Math.},
      volume={185},
      number={1},
       pages={75\ndash 119},
}

\bib{holowinsky2010sieving}{article}{
      author={Holowinsky, R.},
       title={Sieving for mass equidistribution},
        date={2010},
     journal={Ann. of Math. (2)},
       pages={1499\ndash 1516},
}

\bib{holowinsky2010mass}{article}{
      author={Holowinsky, R.},
      author={Soundararajan, K.},
       title={Mass equidistribution for {H}ecke eigenforms},
        date={2010},
     journal={Ann. of Math. (2)},
       pages={1517\ndash 1528},
}

\bib{ingremeau2021local}{article}{
      author={Ingremeau, M.},
       title={Local weak limits of {L}aplace eigenfunctions},
        date={2021},
     journal={Tunisian Journal of Mathematics},
      volume={3},
      number={3},
       pages={481\ndash 515},
}

\bib{knowles2013eigenvector}{article}{
      author={Knowles, A.},
      author={Yin, J.},
       title={Eigenvector distribution of {W}igner matrices},
        date={2013},
     journal={Probab. Theory Related Fields},
      volume={155},
      number={3--4},
       pages={543\ndash 582},
}

\bib{lindenstrauss2006invariant}{article}{
      author={Lindenstrauss, E.},
       title={Invariant measures and arithmetic quantum unique ergodicity},
        date={2006},
     journal={Ann. of Math. (2)},
       pages={165\ndash 219},
}

\bib{marcinek2020high}{article}{
      author={Marcinek, J.},
      author={Yau, H.-T.},
       title={High {D}imensional {N}ormality of {N}oisy {E}igenvectors},
        date={2022},
        ISSN={0010-3616},
     journal={Comm. Math. Phys.},
      volume={395},
      number={3},
       pages={1007\ndash 1096},
}

\bib{porter1956fluctuations}{article}{
      author={Porter, C.~E.},
      author={Thomas, R.~G.},
       title={Fluctuations of nuclear reaction widths},
        date={1956},
     journal={Physical Review},
      volume={104},
      number={2},
       pages={483},
}

\bib{rudnick1994behaviour}{article}{
      author={Rudnick, Z.},
      author={Sarnak, P.},
       title={The behaviour of eigenstates of arithmetic hyperbolic manifolds},
        date={1994},
     journal={Comm. Math. Phys.},
      volume={161},
      number={1},
       pages={195\ndash 213},
}

\bib{shnirel1974ergodic}{article}{
      author={Shnirelman, A.},
       title={Ergodic properties of eigenfunctions},
        date={1974},
     journal={Russian Math. Surveys},
      volume={29},
      number={6},
       pages={181\ndash 182},
}

\bib{tao2012random}{article}{
      author={Tao, T.},
      author={Vu, V.},
       title={Random matrices: universal properties of eigenvectors},
        date={2012},
     journal={Random Matrices Theory Appl.},
      volume={1},
      number={01},
       pages={1150001},
}

\bib{vu2015random}{article}{
      author={Vu, V.},
      author={Wang, K.},
       title={Random weighted projections, random quadratic forms and random
  eigenvectors},
        date={2015},
        ISSN={1042-9832},
     journal={Random Structures Algorithms},
      volume={47},
      number={4},
       pages={792\ndash 821},
}

\bib{wigner1957report}{article}{
      author={Wigner, E.},
       title={Gatlinberg conference on nneutron physics by time-of-flight},
        date={1957},
     journal={Report ORNL},
      volume={2309},
       pages={55},
}

\bib{zelditch1987uniform}{article}{
      author={Zelditch, S.},
       title={Uniform distribution of eigenfunctions on compact hyperbolic
  surfaces},
        date={1987},
     journal={Duke Math. J.},
      volume={55},
      number={4},
       pages={919\ndash 941},
}

\end{biblist}
\end{bibdiv}
\bibliographystyle{abbrv}
\end{document}